\documentclass[11pt]{article}

\usepackage{amsmath,amssymb,graphicx,color,amsthm}
\usepackage[linesnumbered,ruled]{algorithm2e}
\usepackage{tikz}

\usepackage{mathtools}
\usetikzlibrary{shapes.geometric, arrows}
\usepackage{algorithmic}
\usepackage[english]{babel}
\usepackage[utf8]{inputenc}
\usepackage{subcaption,afterpage} 
\usepackage{mathtools}
\usepackage{amsfonts}
\usepackage[colorinlistoftodos]{todonotes}
\usepackage{enumerate}
\usepackage[round]{natbib}
\usepackage{hyperref}
\newcommand{\excise}[1]{}

\newcommand{\parallelsum}{\mathbin{\!/\mkern-5mu/\!}}

\newcommand\RR{\mathbb{R}}

\newtheorem{theorem}{Theorem}
\newtheorem{definition}{Definition}

\newtheorem{proposition}{Proposition}
\newtheorem{corollary}{Corollary}

\newcommand{\RNum}[1]{\uppercase\expandafter{\romannumeral #1\relax}}

\begin{document}

	\title{\mbox{}\\[-11ex]Geodesic Distance Estimation with Spherelets}
	\vspace{-5ex}
	\author{\\[1ex]Didong Li$^1$ and David B Dunson$^{1,2}$ \\ 
		{\em Department of Mathematics$^1$ and Statistical Science$^2$, Duke University}}
	\date{\vspace{-5ex}}
	
	\maketitle
	
Many statistical and machine learning approaches rely on pairwise distances between data points.  The choice of distance metric has a fundamental impact on  performance of these procedures, raising questions about how to appropriately calculate distances.  When data points are real-valued vectors, by far the most common choice is the Euclidean distance. This article is focused on the problem of how to better calculate distances taking into account the intrinsic geometry of the data, assuming data are concentrated near an unknown subspace or manifold.  The appropriate geometric distance corresponds to the length of the shortest path along the manifold, which is the geodesic distance.  When the manifold is unknown, it is challenging to accurately approximate the geodesic distance.  Current algorithms are either highly complex, and hence often impractical to implement, or based on simple local linear approximations and shortest path algorithms that may have inadequate accuracy.  We propose a simple and general alternative, which uses pieces of spheres, or spherelets, to locally approximate the unknown subspace and thereby estimate the geodesic distance through paths over spheres.  Theory is developed showing lower error for many manifolds, with applications in clustering, conditional density estimation and mean regression. The conclusion is supported through multiple simulation examples and real data sets.
\vspace{0.5cm}

	
	\section{Introduction}
\label{sec:intro}
Distance metrics provide a key building block of a vast array of statistical procedures, ranging from clustering to dimensionality reduction and data visualization. Indeed, one of the most common representations of a data set $\{ x_i \}_{i=1}^n$, for $x_i \in \mathcal{X} \subset \RR^D$, is via a matrix of pairwise distances between each of the data points.  The key question that this article focuses on is how to represent distances between data points $x$ and $y$ in a manner that takes into account the intrinsic geometric structure of the data.  Although the standard choice in practice is the Euclidean distance, this choice implicitly assumes that the data do not have any interesting nonlinear geometric structure in their support.  In the presence of such structure, Euclidean distances can provide a highly misleading representation of how far away different data points are.  

This issue is represented in Figure \ref{fig:0}, which shows toy data sampled from a density concentrated close to an Euler spiral.  It is clear that many pairs of points that are close in Euclidean distance are actually far away from each other if one needs to travel between the points along a path that does not cross empty regions across which there is no data but instead follows the `flow' of the data.  As a convenient, if sometimes overly-simplistic, mathematical representation to provide a framework to address this problem, it is common to suppose that the support $\mathcal{X}=\mathcal{M}$, with $\mathcal{M}$ corresponding to a $d$-dimensional Riemannian manifold.  For the data in Figure \ref{fig:0}, the manifold $\mathcal{M}$ corresponds to the $d=1$ dimensional curve shown with a solid line; although the data do not fall exactly on $\mathcal{M}$, we will treat such deviations as measurement errors that can be adjusted for statistically in calculating distances.  

\begin{figure}[htbp]
	\centering
	\includegraphics[width=0.5\textwidth, height=0.3\textheight]{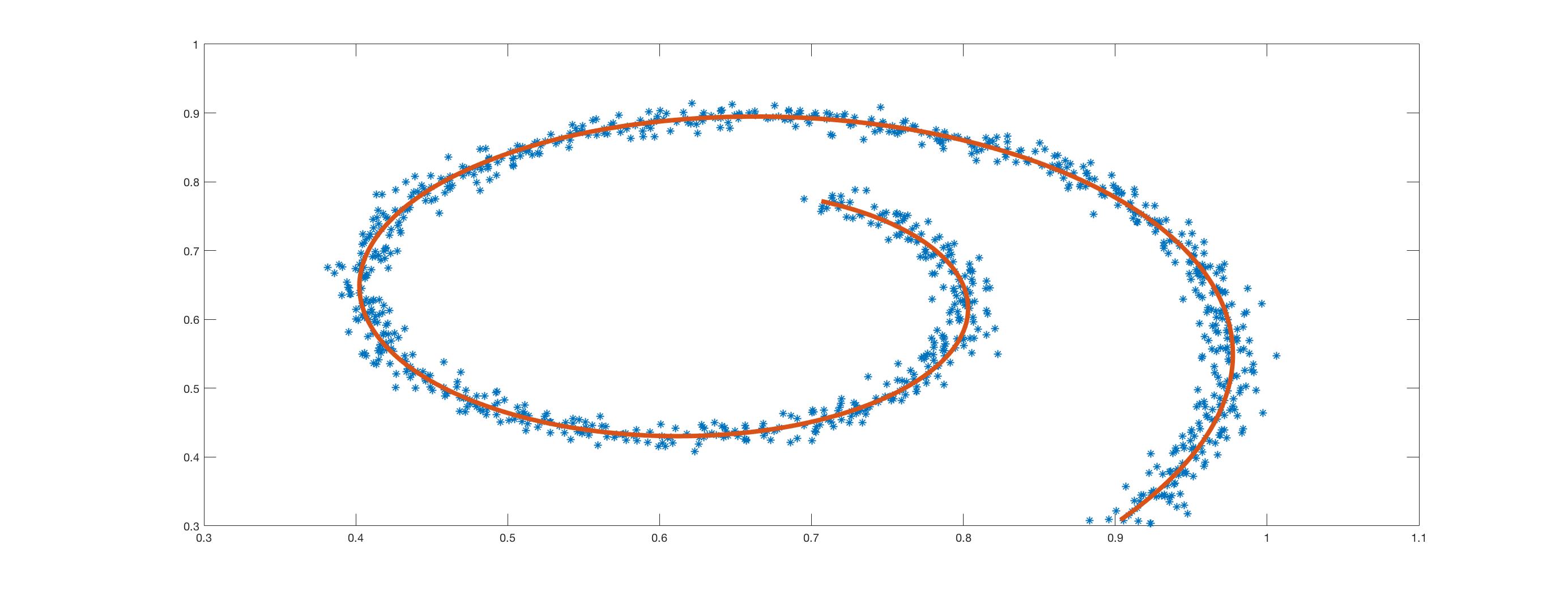}
	\caption[Noisy Euler Spiral]{Noisy Euler Spiral.}\label{fig:0}
	
\end{figure}


The shortest path between two points $x$ and $y$ that both lie on a manifold $\mathcal{M}$ is known as the geodesic, with the length of this path corresponding to the geodesic distance.  If $x$ and $y$ are very close to each other, then the Euclidean distance provides an accurate approximation to the geodesic distance but otherwise, unless the manifold has very low curvature and is close to flat globally, Euclidean and geodesic distances can be dramatically different.  The accuracy of Euclidean distance in small regions has been exploited to develop algorithms for approximating geodesic distances via graph distances.  Such approaches define a weighted graph in which edges connect neighbors and weights correspond to the Euclidean distance.  The estimated geodesic distance is the length of the shortest path on this graph; for details, see \cite{isomap2000} and \cite{silva2003global}.  There is a rich literature considering different constructions and algorithms for calculating the graph distance including \cite{geodist2008}, \cite{geodist2007} and \cite{geodist2004}.
In using the Euclidean distance within local neighborhoods, one needs to keep neighborhoods small to control the global approximation error.  This creates problems when the sample size $n$ is not sufficiently large and when the density $\rho$ of the data points is not uniform over $\mathcal{M}$ but instead is larger in certain regions than others. 

A good strategy for more accurate geodesic distance estimation is to improve the local Euclidean approximation while continuing to rely on graph distance algorithms. A better local approximation leads to better global approximation error.  This was the focus of a recent local geodesic distance estimator proposed in 
\cite{wu2018think} and \cite{malik2019connecting}.  Their covariance-corrected estimator adds an adjustment term to the Euclidean distance, which depends on the projection to the normal space.  This provides a type of local adjustment for curvature, and they provide theory on approximation accuracy.  However, their approach often has poor empirical performance in our experience, potentially due to statistical inaccuracy in calculating the adjustment and to lack of robustness to measurement errors.

We propose an alternative local distance estimator, which has the advantage of providing a simple and transparent modification of Euclidean distance to incorporate curvature.  This is accomplished by approximating the manifold in a local neighborhood using a sphere, an idea proposed in \cite{spherelets} but for manifold learning and not geodesic distance estimation.  Geodesic distance estimation involves a substantially different goal, and distinct algorithms and theory need to be developed.  The sphere has the almost unique features of both accounting for non-zero curvature and having the geodesic distance between any two points in a simple closed form; even for simple manifolds the geodesic is typically intractable.  We provide a transparent and computationally efficient algorithm, provide theory justifying accuracy and show excellent performance in a variety of applications including clustering, conditional density estimation and mean regression on multiple real data sets.

\section{Methodology}
Throughout this paper, we assume $M$ is a smooth compact Riemannian manifold with Riemannian metric $g$. Letting $\gamma(s)$ be a geodesic in arc length parameter $s$, the geodesic distance $d_M$ is defined by 
$$d_M(x,y)\coloneqq \inf \left\{L(\gamma)\mid\gamma(0) = x,\ \gamma(S)=y\right\},$$
where $L(\gamma)\coloneqq \int_0^S g\{\gamma'(s),\gamma'(s)\}^{1/2}\mathrm{dt}$ is the length of curve $\gamma$. Given points $X=\{x_1,\cdots,x_n\}$ on the manifold, the goal is to estimate the pairwise distance matrix $GD\in\RR^{n\times n}$ where $GD_{ij}=d_M(x_i,x_j)$. First we propose a local estimator, that is, to estimate $GD_{ij}$  where $x_i$ and $x_j$ are close to each other, and then follow the local-to-global philosophy to obtain a global estimator, for arbitrary $x_i$ and $x_j$.

\subsection{Local Estimation}\label{local_dist}
In this subsection, we focus on geodesic distance estimation between neighbors. The simplest estimator of $d_M(x_i,y_i) $ is $\|x_i-x_j\|$, denoted by $d_E(x_i,x_j)$. However, the estimation error of the Euclidean distance depends on the curvature linearly. As a result, a nonlinear estimator incorporating curvature needs to be developed to achieve a smaller estimation error for curved manifolds. We propose a nonlinear estimator using spherical distance, which is motivated by the fact that osculating circles/spheres approximate the manifold better than tangent lines/spaces. On the osculating sphere, the geodesic distance admits an analytic form, which we use to calculate local geodesic distances.

Let $S_{x_i}(V,c,r)$ be a $d$ dimensional sphere centered at $c$ with radius $r$ in $d+1$ dimensional affine space $x_i+V$, approximating $M$ in a local neighborhood of $x_i$. Letting $\pi$ be the orthogonal projection from the manifold to the sphere, the spherical distance is defined as
\begin{eqnarray}
d_S(x_i,x_j) &\coloneqq & r \arccos\left\{ \frac{\pi(x_i)-c}{r}\cdot \frac{\pi(x_j)-c}{r} \right\}, \label{eq:dS}
\end{eqnarray}
the geodesic distance between $\pi(x_i)$ and $\pi(x_j)$ on the sphere $S_{x_i}(V,c,r)$. The spherical distance depends on the choice of sphere $S_{x_i}(V,c,r)$, which will be discussed in section \ref{CSPCA}.

\subsection{Global Estimation}\label{global_dist}

We now consider  global estimation of the geodesic distance $d_M(x_i,x_j)$ for any $x_i,x_j$. The popular Isomap algorithm was proposed in  \cite{isomap2000} for dimension reduction for manifolds isometrically embedded in higher dimensional Euclidean space. Isomap relies on estimating the geodesic distance using the graph distance based on a local Euclidean estimator.  Let $G$ be the graph with vertices $x_i$. For any two points $x_i$ and $x_j$ that are close to each other, Isomap estimates $d_M(x_i,x_j)$ using $\|x_i-x_j\|$.  This leads to the following global estimator of $d_M(x_i,x_j)$, for any two points 
$x_i,x_j\in X$,
\begin{eqnarray}
d_{EG}(x_i,x_j) & \coloneqq & \min _P\sum_{l=0}^{p-1}\|x_{i_l}-x_{i_{l+1}}\|, \label{eq:isomapGD}
\end{eqnarray}
where $P$ varies over all paths along $G$ having $x_{i_0}=x_i$ and $x_{i_p}=x_j$.  In particular, the global distance is defined by the length of the shortest path on the graph, where the length of each edge is given by the Euclidean distance.  In practice, local neighbors are determined by a $k$-nearest neighbors algorithm, with the implementation algorithm given in Section \ref{Algorithms}.

The estimator in expression (\ref{eq:isomapGD}) has been successfully implemented in many different contexts.  However, the use of a local Euclidean estimator $\|x_{i_l}-x_{i_{l+1}}\|$ is a  limitation, and one can potentially improve the accuracy of the estimator by using a local approximation that can capture curvature, such as $d_S(x_i,x_j)$ in (\ref{eq:dS}).
This leads to the following alternative estimator: 
\begin{eqnarray}
d_{SG}(x,y)\coloneqq\min _P\sum_{l=0}^{p-1}d_S(x_{i_l}-x_{i_{l+1}}), \label{eq:sphereletGD}
\end{eqnarray}
where $P$ is as defined for (\ref{eq:isomapGD}) and an identical graph paths algorithm can be implemented as for Isomap, but with spherical distance used in place of Euclidean distance in the local component.

\subsection{Osculating Sphere}\label{CSPCA}

In order to calculate the local spherical distances necessary for computing (\ref{eq:sphereletGD}), we first need to estimate `optimal' approximating spheres within each local neighborhood, characterized by the $k$ nearest neighbors of $x_i$, denoted by $X_i^{[k]}$. The local sample covariance matrix is defined as $\Sigma_{k}(x_i) = k^{-1}\sum_{x_j\in X_i^{[k]}}(x_j-x_i)(x_j-x_i)^T$.  The eigen-space spanned by the first $d+1$ eigenvectors of $\Sigma_{k}(x_i)$, denoted by $V^*=\mathrm{span}\left[\mathrm{evec}_{1}\{\Sigma_k(x_i)\},\cdots,\mathrm{evec}_{d+1}\{\Sigma_k(x_i)\}\right]$ is the best estimator of the $d+1$ dimensional subspace $V$. Here we are ordering the eigenvectors by the corresponding eigenvalues in decreasing order.

Observe that the target sphere $S_{x_i}(V^*,c^*,r^*)$ passes through $x_i$ so we have $r^* = \|c^*-x_i\|$.  Hence, the only parameter to be determined is $c^*$ and then $r^*=\|c^*-x_i\|$. To estimate $c^*$, we propose a centered $k$-osculating sphere algorithm. Suppose $x_j\in S_{x_i}(V^*,c^*,r^*)$, then the projection of $x_j$ to $x_i+V^*$, denoted by $y_j=x_i+V^*V^{*\top}(x_j-x_i)$, is among the zeros of the function $\|y-c^*\|^2-r^{*2}$ where $r=\|c^*-x_i\|=\|c^*-y_i\|$. We use this to define a loss function for estimating $c$ in Definition \ref{def:CSPCA}; related `algebraic' loss functions were considered in \cite{coope1993circle} and \cite{spherelets}.

\begin{definition}\label{def:CSPCA}
	Under the above assumptions and notations, let $c^*$ be the minimizer of the following optimization problem:
	\begin{equation}\label{opt}
	\underset{c}{\arg\min}\ \sum_{x_j\in X^{[k]}_i}\left(\|y_j-c\|^2-\|y_i-c\|^2\right)^2.
	\end{equation}
	Letting $r^*=\|x_i-c^*\|$, the sphere $S_{x_i}(V^*,c^*,r^*)$ is called the centered $k$-osculating sphere of $X$ at $x_i$. 
\end{definition}

We can tell from the definition that the centered sphere is a nonlinear analogue of centered principal component analysis to estimate the tangent space. There is one additional constraint for the centered $k$-osculating sphere: the sphere passes through $x_i$. This constraint is motivated by the proof of Theorem \ref{thm:general}, see the supplementary materials. 

Observe that the optimization problem is convex with respect to $c$ and we can derive a simple analytic solution, presented in the following theorem.

\begin{theorem}\label{thm:CSPCA}
	The minimizer of the optimization problem (\ref{opt}) is given by:
	$$c^*= \frac{1}{2}H^{-1}f,$$
	where $H =\sum_{x_j\in X^{[k]}_i} (y_j-y_i)(y_j-y_i)^\top$ and $f=\sum_{x_j\in X^{[k]}_i}(\|y_j\|^2-\|y_i\|^2)(y_j-y_i)$.

\end{theorem}

\subsection{Algorithms}\label{Algorithms}
In this subsection, we present algorithms to calculate the spherical distance. Before considering algorithms for distance estimation, we present the algorithm for the centered $k$-osculating sphere, shown in Algorithm \ref{alg:CSPCA}.

In real applications where the data are noisy, we recommend replacing the centered $k$-osculating sphere by an uncentered version because in this case the base point $x$ may not be on the manifold so shifting toward $x$ can negatively impact the performance. In addition, the constraint $r=\|x_i-c\|$ restricts the degrees of freedom when choosing the optimal $r$.  The only difference is that instead of centering at the base point $x$ and forcing $r=\|x_i-c\|$, we instead shift $x_{i}$ to the mean $\bar x=\frac{1}{n}\sum_{i=1}^n x_i$ and average $\|x_j-c\|$, as shown in Algorithm \ref{alg:SPCA}.

\hspace{-5pt} \begin{minipage}[t]{0.53\textwidth}
	\null 
	\begin{algorithm}[H]
		\SetKwData{Left}{left}\SetKwData{This}{this}\SetKwData{Up}{up}
		\SetKwFunction{Union}{Union}\SetKwFunction{FindCompress}{FindCompress}
		\SetKwInOut{Input}{input}\SetKwInOut{Output}{output}
		\Input{Data set $\{x_i\}_{i=1}^n$, base point $x$, manifold dimension $d$.}
		\Output{Sphere $S_x(V,c,r)$.}
		\BlankLine
		
		\emph{$\Sigma = \frac{1}{n}\sum_{i=1}^n (x_i-x)(x_i-x)^\top$}\;
		\emph{$V=\mathrm{span}\{\mathrm{evec}_1(\Sigma),\cdots,\mathrm{evec}_{d+1}(\Sigma)\}$}\;
		\emph{$y_i = x+VV^\top(x_i-x)$, $y=x$}\;
		\emph{$H=\sum_{i=1}^n(y_i-y)(y_i-y)^\top$}\;
		\emph{$f=\sum_{i=1}^n(\|y_i\|^2-\|y\|^2)(y_i-y)$}\;
		\emph{$c = \frac{1}{2}H^{-1}f$}\;
		\emph{$r = \|x-c\|$.}
		\caption{\hspace{-0pt }{\small Centered $k$-osculating sphere}}
		\label{alg:CSPCA}
	\end{algorithm}
\end{minipage}%
\begin{minipage}[t]{0.48\textwidth}
	\null
	\begin{algorithm}[H]
		\SetKwData{Left}{left}\SetKwData{This}{this}\SetKwData{Up}{up}
		\SetKwFunction{Union}{Union}\SetKwFunction{FindCompress}{FindCompress}
		\SetKwInOut{Input}{input}\SetKwInOut{Output}{output}
		\Input{Data set $\{x_i\}_{i=1}^n$, manifold dimension $d$.}
		\Output{Sphere $S_x(V,c,r)$.}
		\BlankLine
		
		\emph{$\Sigma = \frac{1}{n}\sum_{i=1}^n (x_i-\bar x)(x_i-\bar x)^\top$}\;
		\emph{$V=\mathrm{span}\{\mathrm{evec}_1(\Sigma),\cdots,\mathrm{evec}_{d+1}(\Sigma)\}$}\;
		\emph{$y_i = \bar x+VV^\top(x_i-\bar x)$}\;
		\emph{$H=\sum_{i=1}^n(y_i-\bar y)(y_i-\bar y)^\top$}\;
		\emph{$f=\sum_{i=1}^n\left(\|y_i\|^2-\frac{1}{n}\Sigma_{j} \|y_j\|^2\right)(y_i-\bar y)$}\;
		\emph{$c = \frac{1}{2}H^{-1}f$}\;
		\emph{$r = \frac{1}{n}\sum_{i=1}^n\|y_i-c\|$.}
		
		\caption{{\small $k$-osculating sphere}}
		\label{alg:SPCA}
	\end{algorithm}
	
\end{minipage}


\begin{algorithm}
	\SetKwData{Left}{left}\SetKwData{This}{this}\SetKwData{Up}{up}
	\SetKwFunction{Union}{Union}\SetKwFunction{FindCompress}{FindCompress}
	\SetKwInOut{Input}{input}\SetKwInOut{Output}{output}
	\Input{Data set $\{x_i\}_{i=1}^n$; tuning parameters $k,\ d$.}
	\Output{Pairwise distance matrix $SD$}
	\BlankLine
	\emph{Initialize $SD\in\RR^{n\times n}$ where $SD_{ij}=\infty$ for $i\neq j$ and $SD_{ii}=0$}\;
	\For{$i=1:n$}{
		\emph{Find $K$ nearest neighbors of $x_i$, $X_i^{[k]}=\{x_{i_1},\cdots,x_{i_k}\}$}\;
		\emph{Calculate the $p$-dimensional spherical approximation of $X_i^{[k]}$, denoted by $S^i_x(V_i,c_i,r_i)$} given by Algorithm \ref{Alg:CSPCA}\;
		\For{$j = 1:k$}{
			\emph{Calculate the projection of $x_{i_j}$ to the sphere $S^i_x(V_i,c_i,r_i)$, $\widehat{x}_{i_j}=c_i+\frac{r_i}{\|V_iV_i^\top(x_{i_j}-c_i)\|}V_iV_i^\top(x_{i_j}-c_i)$}\;
			\emph{$SD_{i i_j}=r_i\arccos\left(\frac{x_i-c_i}{r_i}\cdot \frac{x_{i_j}-c_i}{r_i}\right)$}
		}
	}
	\emph{Symmetrization: $SD= \frac{SD+SD^\top}{2}$.}
	\caption{Local Spherical Distance }
	\label{alg:local}
\end{algorithm}

From Algorithm \ref{alg:local} we obtain the local pairwise distance matrix $SD$, where $SD_{ij}$ denotes the distance between $x_i$ and $x_j$. However, if $x_i$ and $x_j$ are not neighbors of each other, the distance will be infinity, or equivalently speaking there is no edge connecting $x_i$ and $x_j$ in graph $G$. Then we need to convert the local distance to global distances by the graph distance proposed in Section \ref{global_dist}. There are multiple algorithms for shortest path search on graphs including the Floyd-Warshall algorithm (\cite{floyd1962algorithm} and \cite{warshall1962theorem}) and Dijkstra's algorithm (\cite{dijkstra1959note}); here we adopt the Dijkstra's algorithm, which is easier to implement. Algorithm \ref{alg:global} shows how to obtain the graph spherical distance from local spherical distance. 

\begin{algorithm}
	\SetKwData{Left}{left}\SetKwData{This}{this}\SetKwData{Up}{up}
	\SetKwFunction{Union}{Union}\SetKwFunction{FindCompress}{FindCompress}
	\SetKwInOut{Input}{input}\SetKwInOut{Output}{output}
	\Input{Local pairwise distance matrix $SD\in \RR^{n\times n}$.}
	\Output{Graph pairwise distance matrix $SD$.}
	\BlankLine
	\For{$i=1:n$}{
		\emph{$SD = min[SD , rep\{SD(:,i), n\} + rep\{SD(i,:) , n\}]$, where $rep(v,n)$ is $n$ copies of row/column vector $v$.}
	}
	\caption{Graph Spherical Distance}
	\label{alg:global}
\end{algorithm}

We note that in the local estimation, the computational complexity for $d_M(x_i,x_j)$ is $O(\min\{k,D\}^3)$, where $k$ is assumed to be much smaller than $n$. To compare with, the computational complexity of $d_E(x_i,x_j)$ is $O(D)$. Hence, in general, we are not introducing more computation cost by replacing the local Euclidean distance by the local spherical distance unless $d$ is not very small relative to $D$. Once the graph is determined, the computational complexity of Dijkstra's algorithm is $O(n^2)$, where $n$ is the sample size, and this complexity does not depend on which  local distance is applied to obtain the weights on the graph $G$. Hence, the total computational complexity for the graph Euclidean distance estimator is $O\left(nkD+n^2\right)$ while the complexity for the graph spherical distance estimator is $O\left(n\min\{k,D\}^3+n^2\right)$.
\section{Error Analysis}
In this section, we analyze why the spherical distance is a better estimator than the Euclidean distance from a theoretical perspective following a local-to-global philosophy. 
\subsection{Local Error}
First we study the local error, that is, $|d_S(x,y)-d_M(x,y)|$ for $y\in B_{\bar r}(x)$, where $B_{\bar r}(x)$ is the geodesic ball on $M$ centered at $x$ with radius $\bar r$. It is well known that the error of the Euclidean estimator is third order, as formalized in the following proposition.
\begin{proposition}\label{Euc_error}
	Assume $d_M(x,y) = s$, then 
	$$s-\frac{s^3}{24r_0^2} \leq d_E(x,y)\leq s ,$$
	with $\frac{1}{r_0 }= \sup\left\{\|\gamma''(s)\|\right\}$, where $\gamma$ varies among all geodesics on $M$ in arc length parameter. In terms of the error rate,  $d_E(x,y)=s+O(s^3)$.
\end{proposition}
These bounds are tight and the proof can be found in \cite{smolyanov2007chernoff}.
The Euclidean distance is a simple estimator of the geodesic distance, and the error is $\frac{s^3}{24r_0^2}$. 
While this may seem to be a good result, if the manifold has high curvature, so that $r_0$ is very small, performance is not satisfactory.  This is implied by the $r_0^{-2}$ multiple on the error rate, and is also clearly apparent in experiments shown later in the paper.  

Now we consider the error of the spherical distance proposed in section \ref{local_dist}. For simplicity, we first consider the case in which $M=\gamma$ is a curve in $\RR^2$ with domain $[0,S]$. Without loss of generality, fix $x=\gamma(0)$ but vary $y=\gamma(s)$. Let  ${\bf n}$ be the unit normal vector of $\gamma$ at $x$, that is $\gamma''(0)=\kappa {\bf n}$. Let $r = \frac{1}{|\kappa|}$ and $c=x-\frac{1}{\kappa}\boldsymbol{n}$, which determine a circle $C_x(c,r)$ centered at $c$ with radius $r$. This circle $C_x(c,r)$ is called the osculating circle of the curve $\gamma$, which is the ``best" circle  approximation to the curve.  Letting $\pi: \gamma\rightarrow C_x(c,r)$ be the projection to the osculating circle, the error in $d_S(x,y)$ as an estimator of $d_M(x,y)$ is shown in the following theorem.
\begin{theorem}\label{thm:curve}
	Let $x=\gamma(0)$ and $y=\gamma(s)$, so $d_M(x,y)=s$, then
	$$d_S(x,y)=s+O(s^4).$$ 
\end{theorem}
Comparing to the error of Euclidean estimation in Proposition \ref{Euc_error}, the spherical estimate improves the error rate from $O(s^3)$ to $O(s^4)$.

The above result is for curves, and as a second special case we suppose that $M^d\subset \RR^{d+1}$ is a $d$ dimensional hyper-surface. Similar to the curve case, the spherical distance can be defined on any sphere $S_x(c,r)$ passing through $x$ with center $c$ and radius $r$ where $c=x-\frac{1}{\kappa}\boldsymbol{n}$ and $\boldsymbol{n}$ is the normal vector of the tangent space $T_xM$, $r=\frac{1}{|\kappa|}$. However, for geodesics along different directions, denoted by $\gamma_v\coloneqq \exp_x(sv)$ where $v\in UT_xM$, the curvature $\kappa_v(x)$ defined by  $\gamma_v''(0) = \kappa_v(x) \boldsymbol{n}$ might be different. Let $\kappa_{2}(x) = \sup_{v\in UT_xM}\kappa_v(x)$ and $\kappa_{1}(x) = \inf_{v\in UT_xM}\kappa_v(x)$, where the maximum and minimum can be achieved due to the compactness of $UT_xM$. Fix any $\kappa_0(x)\in[\kappa_{1}(x),\kappa_{2}(x)]$, let $S_x(c,r)$ be the corresponding sphere, and $\pi: M\rightarrow S_x(c,r)$ be the projection. The estimation error is given by the following theorem.
\begin{theorem}\label{thm:hypersurf}
	Fix $x\in M$, for $y=\exp_x(sv)$ such that $d_M(x,y)=s$, let $\kappa_y=\kappa_v(x)$, then the estimation error of spherical distance is given by 
	$$d_S(x,y)=s+(\kappa_y-\kappa_0)(\kappa_y-2\kappa_0)s^3+O(s^4).$$ 
\end{theorem}
In the worst case, the error has the same order as that for the Euclidean distance. However, there are multiple cases where the error is much smaller than the Euclidean one, shown in the following corollary.

\begin{corollary}	\label{cly:hypersurf}
	Under the same conditions in Theorem \ref{thm:hypersurf},
	\begin{enumerate}[(1)] 
		\item If $\kappa_y=\kappa_0$ or $\kappa_y = 2 \kappa_0$, then $d_S(x,y)=s+O(s^4)$.
		\item If $|\kappa_2(x)-\kappa_1(x)|<\bar r$, then $d_S(x,y)=s+O(\bar r s^3)$.
	\end{enumerate}
\end{corollary}
Assume $\kappa_0=\kappa_{v_0}$, then for all $y\in\{\exp_x(sv)\mid |\kappa_v-\kappa_0|\leq \bar r\}$, which is a neighborhood of the geodesic $\exp_x(tv_0)$, spherical estimation outperforms the Euclidean estimation. The closer to the central geodesic, the better the estimation performance. For a point $x$ where $\kappa_v(x)$ is not changing rapidly along different directions, the spherical estimation works well in the geodesic ball $B_{\bar r}(x)$.

Finally we consider the most general case: $M$ is a $d$ dimensional manifold embedded in $\RR^D$ for any $D>d$. Let $S_x(c,r)$ be a $d$ dimensional sphere whose tangent space is also $T_xM$. Letting $\pi$ be the projection to the sphere, the estimation error is given by the following theorem.
\begin{theorem}\label{thm:general}
	Fix $x\in M$, for $y=\exp_x(sv)$ such that $d_M(x,y)=s$, then the estimation error of spherical distance is given by 
	$$d_S(x,y)=s+O(s^3).$$ 
\end{theorem}
Combining Theorem \ref{thm:curve}-\ref{thm:general}, we conclude that spherical estimation is at least the same as Euclidean estimation in terms of the error rate, and in many cases, the spherical estimation outperforms the Euclidean one.

\subsection{Global Error}
In this section we analyze the estimation error: $|d_{SG}(x,y)-d_M(x,y)|$ for any $x,\ y\in M$. The idea is to pass the local error bound to the global error bound. We use the same notation introduced in Section \ref{global_dist}. 
\begin{theorem}\label{thm:globalerror}
	Assume $M$ is a compact, geodesically convex submanifold embedded in $\RR^D$ and $\{x_i\}_{i=1}^n\subset M$ is a set of points, which are vertices of graph $G$. Introduce constants $\epsilon_{\min}>0$, $\epsilon_{\max}>0$, $0<\delta<\epsilon_{\min/4}$ and let $C$ be the constant such that $|d_S(x,y)-d_M(x,y)|\leq d_M(x,y)\{1+Cd^2_M(x,y)\}$ according to Theorem \ref{thm:general}. Suppose
	\begin{enumerate}
		\item $G$ contains all edges $xy$ with $\|x-y\|\leq \epsilon_{\min}$.
		\item All edges $xy$ in G have length $\|x-y\|\leq \epsilon_{\max}$.
		\item $\{x_i\}_{i=1}^n$ is a $\delta$-net of $M$, that is, for any $x\in M$, there exists $x_i$ such that $d_M(x,x_i)\leq \delta$.
	\end{enumerate}
	Then for any $x,y\in M$ 
	$$(1-\lambda_1)d_M(x,y)\leq d_{SG}(x,y)\leq (1+\lambda_2)d_M(x,y),$$
	where $\lambda_1 = C\epsilon_{\max}^2$ and $\lambda_2= \frac{4\delta}{\epsilon_{\min}}+C\epsilon_{\max}^2+\frac{4C\delta\epsilon^2_{\max}}{\epsilon_{\min}}$.
	
\end{theorem}
As the sample size grows to infinity, $\delta, \epsilon_{\min}, \epsilon_{\max}\rightarrow 0$ and we can carefully choose the size of the neighborhood so that $\delta/\epsilon_{\min}\rightarrow 0$. As a result, $\lambda_1, \lambda_2\rightarrow 0$ so $d_S(x,y)\rightarrow d_M(x,y)$ uniformly.

\section{Simulation Studies}
\subsection{Euler Spiral}
We test the theoretical results on generated data from manifolds in which the geodesic distance is known so that we can calculate the error. The first example we consider is the Euler spiral, a curve in $\RR^2$. The Cartesian coordinates are given by Fresnel integrals: $\gamma(s)=\{x(s),y(s)\}$ where 
$$x(s)=\int_0^s\cos(t^2)\mathrm{dt},\quad y(s)=\int_0^s\sin(t^2)\mathrm{dt}.$$
The main feature of the Euler spiral is that the curvature grows linearly, that is, $\kappa(s)=s$. We generate $500$ points uniformly on $[0,2]$. Then we fix $x=\gamma(1.6)$ and choose $\bar{r}=0.04$, so there are $20$ points falling inside the geodesic ball $B_{\bar{r}}(x)$, denoted by $y_1,\cdots,y_{20}$. Then we can calculate the Euclidean $\|y_i-x\|$ and the spherical distance $d_S(x,y_i)$. 

\begin{figure}[h]
	\centering
	\begin{minipage}{0.3\textwidth}
		\centering
		\includegraphics[width=1\textwidth, height=0.2\textheight]{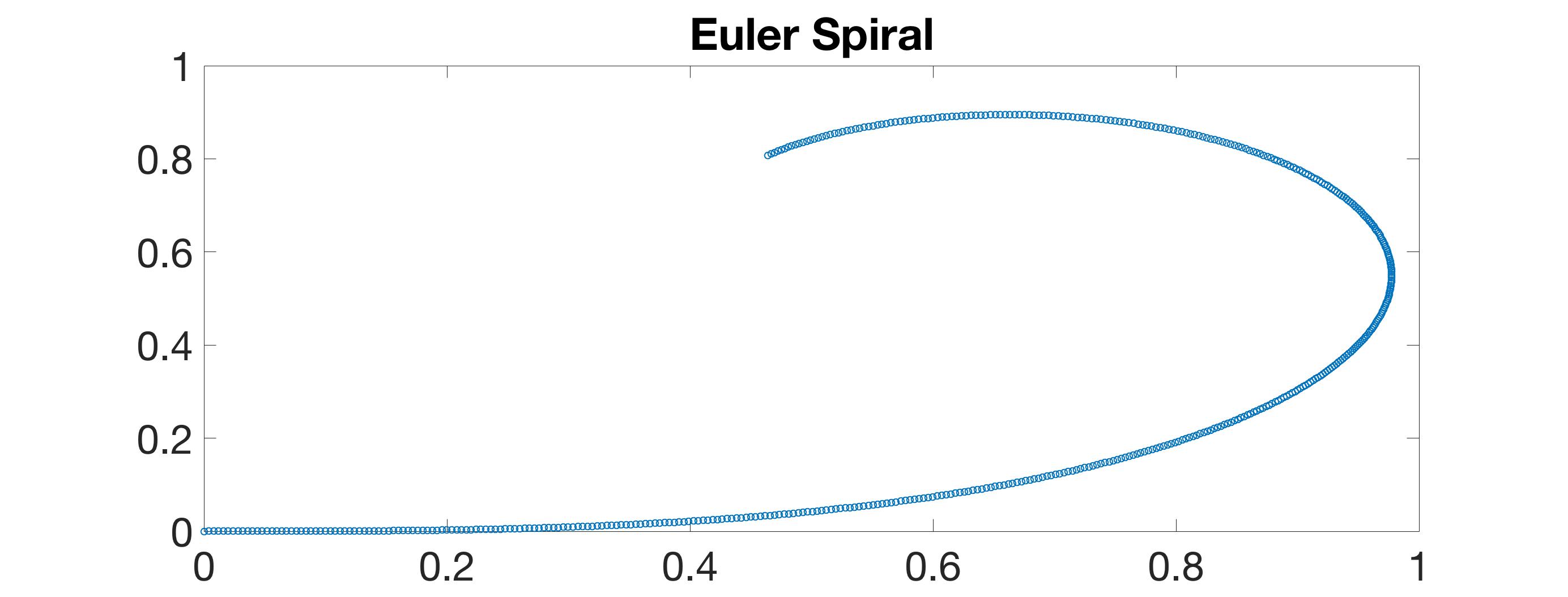}
		\subcaption[Linearly separable classes.]{Euler spiral}\label{fig:1a}
	\end{minipage}%
	\begin{minipage}{0.4\textwidth}
		\centering
		\includegraphics[width=1\textwidth,height=0.2\textheight]{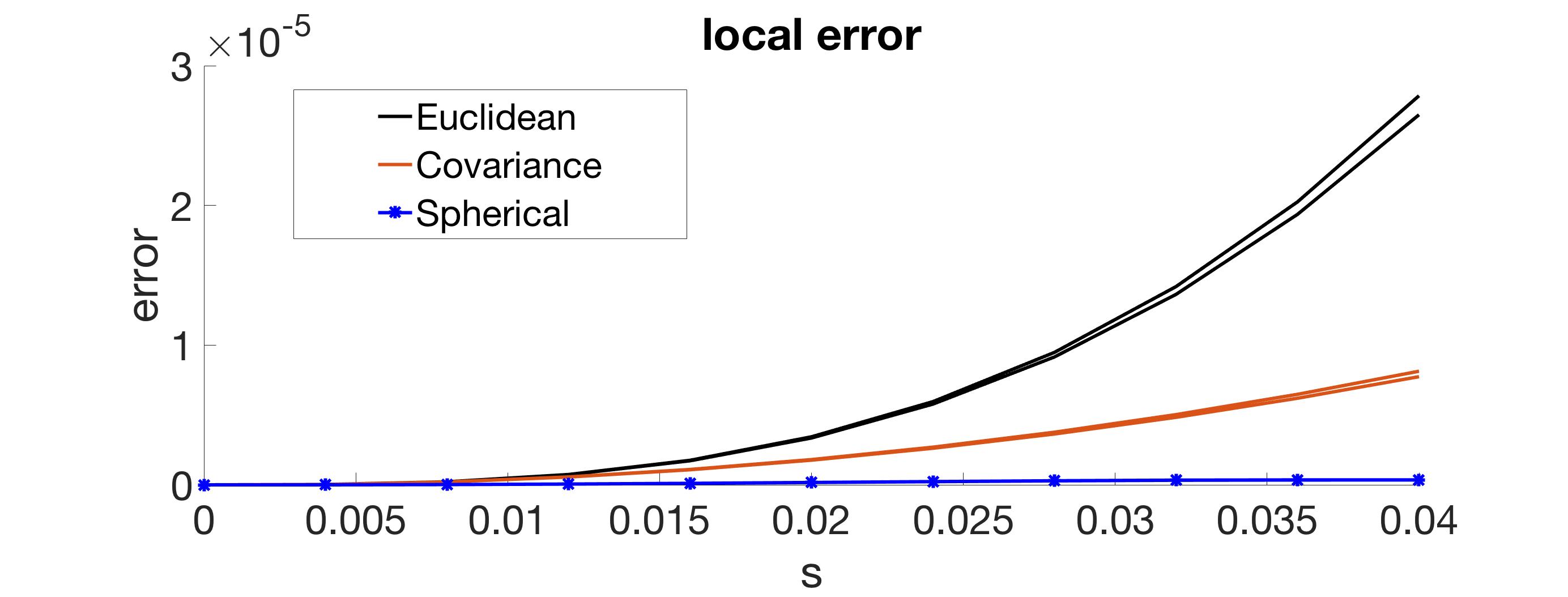}
		\subcaption[Linearly inseparable classes.]{Plot for local error.}\label{fig:1b}
	\end{minipage}%
	\begin{minipage}{0.4\textwidth}
		\centering
		\includegraphics[width=1\textwidth,height=0.2\textheight]{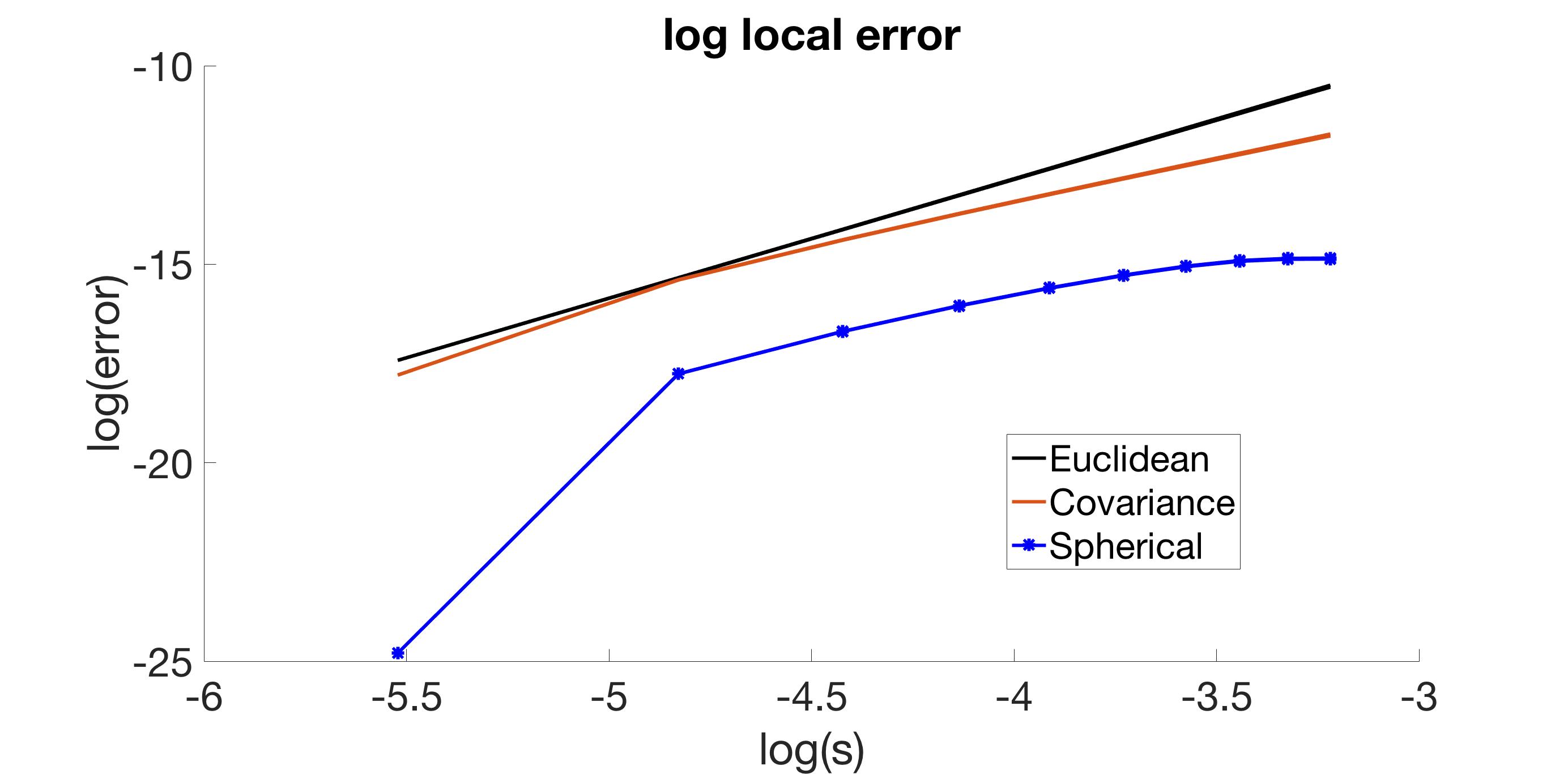}
		\subcaption[Linearly inseparable classes.]{$\log$ plot for local error }\label{fig:1c}
	\end{minipage}%
	\caption{Local error for Euler spiral} \label{fig:EulerSpiral}
\end{figure}

The covariance-corrected geodesic distance estimator (\cite{malik2019connecting}) can be viewed as the state-of-the-art. We compare the spherical distance with both Euclidean distance and the covariance-corrected distance. Figure \ref{fig:1a} is the spiral and Figure \ref{fig:1b} contains the error plot for the three algorithms. To visualize the rate, we also present the $\log-\log$ plot in Figure \ref{fig:1c}. The results match our theory and the spherical estimator has the smallest error among these three algorithms.

Then we consider the global error. By the definition of arc length parameter, the pairwise geodesic distance matrix is given by $GD_{ij}=|s_i-s_j|$. Denote the Euclidean pairwise distance matrix by $D$, the graph distance based on Euclidean distance, covariance-corrected distance and spherical distance by $EG$, $CG$ and $SG$, respectively. As the most natural measurement of the global error, we calculate and compare the following norms:
$$\|GD-D\|,\quad \|GD-EG\|,\quad \|GD-CG\|, \quad\|GD-SG\|.$$
Table \ref{table:error} shows the global error when the total sample size is $500$ and $k$ is chosen to be $3$. Furthermore, we vary the curvature  from $[0,1]$ to $[3,4]$ to assess the influence of curvature on these estimators. 
\begin{table}[h!]
	\centering
	\def~{\hphantom{0}}
	\caption{Global error for Euler spiral}
	\begin{tabular}{ccccc}
		Curvature& D& EG & CG & SG\\
		$[0,1]$ & 3.0829 & 1.4708e-04 & 3.2701e-05 & \bf{3.2291e-07}\\
		$[1,2]$ & 23.9957 & 1.0699e-03 & 2.3807e-04 & \bf{5.5456e-07} \\
		$[2,3]$ & 58.8129 & 2.917e-03 & 6.4931e-04 & \bf{9.2362e-07} \\
		$[3,4]$ & 95.2806 & 5.6887e-03 & 1.2665e-03& \bf{1.2929e-06}
	\end{tabular}
	\label{table:error}
\end{table}

The global Euclidean distance is by far the worst and the graph spherical distance is the best in all cases. Furthermore, as the curvature increases, the spherical error increases the most slowly. This matches the theoretical analysis, since the spherical estimator takes the curvature into consideration. 

In real applications, almost all data contain measurement error, so the data may not exactly lie on some manifold, but instead may just concentrate around the manifold with certain noise. The robustness of the algorithm with respect to the noise is a crucial feature. To assess this, we generate samples from the Euler spiral and add Gaussian noise $\epsilon_i\sim N(0,\sigma^2 Id_D)$ where $\sigma$ is the noise level. In this setting the local error is not very meaningful since $x_i$ is no longer on the manifold. However, the global error is still informative since the pairwise distance matrix contains much information about the intrinsic geometry of the manifold. Since the ground truth $d_M(x_i,x_j)$ is not well defined, we firstly apply the Graph Euclidean distance to a large data set, and treat these results as ground truth $GD$. The reason is that when the sample size is large enough, all the above global estimators converge to the true distance except for $D$. Then we subsample a smaller dataset and apply these global estimators to obtain $EG$, $CG$ and $SG$ and compute the error. We test on different subsample sizes to assess the stability of the algorithms and the performance on small data sets. 

\begin{figure}[htbp]
	\centering
	\begin{minipage}{0.4\textwidth}
		\centering
		\includegraphics[width=1\textwidth, height=0.2\textheight]{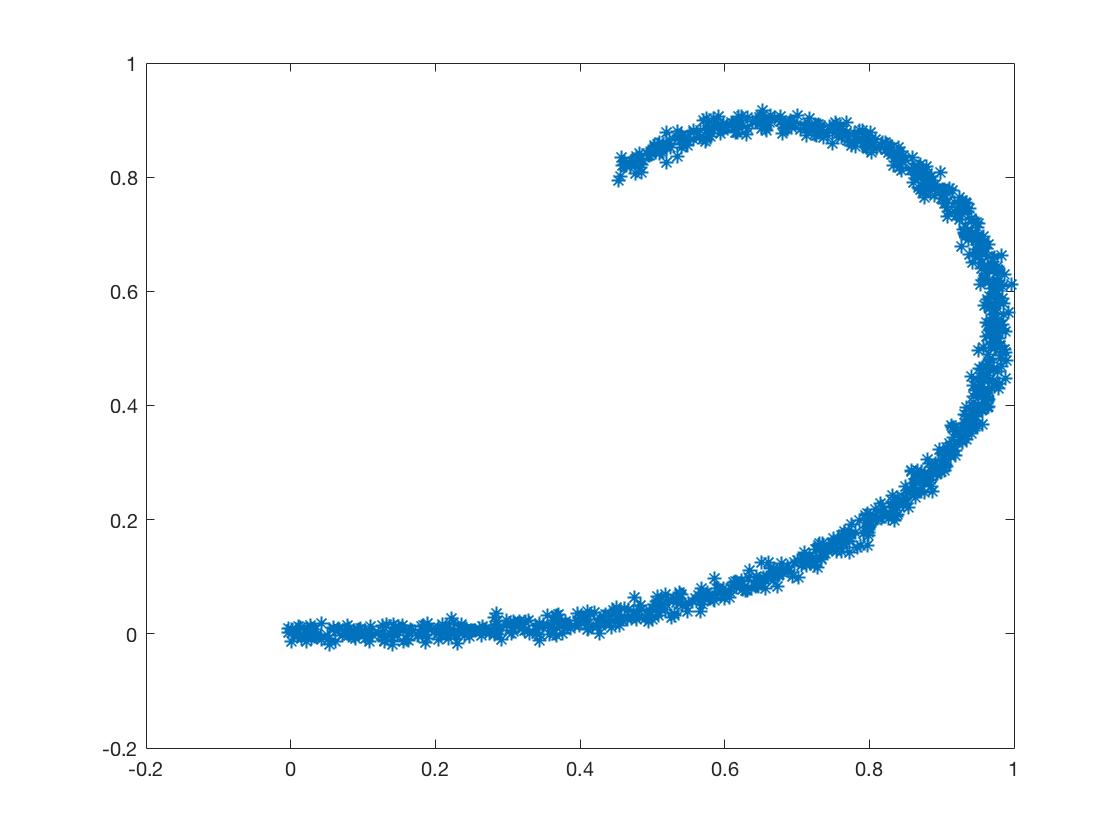}
		\subcaption[Noisy Euler Spiral.]{Noisy Euler Spiral.}\label{fig:2a}
	\end{minipage}%
	\begin{minipage}{0.6\textwidth}
		\centering
		\includegraphics[width=1\textwidth,height=0.2\textheight]{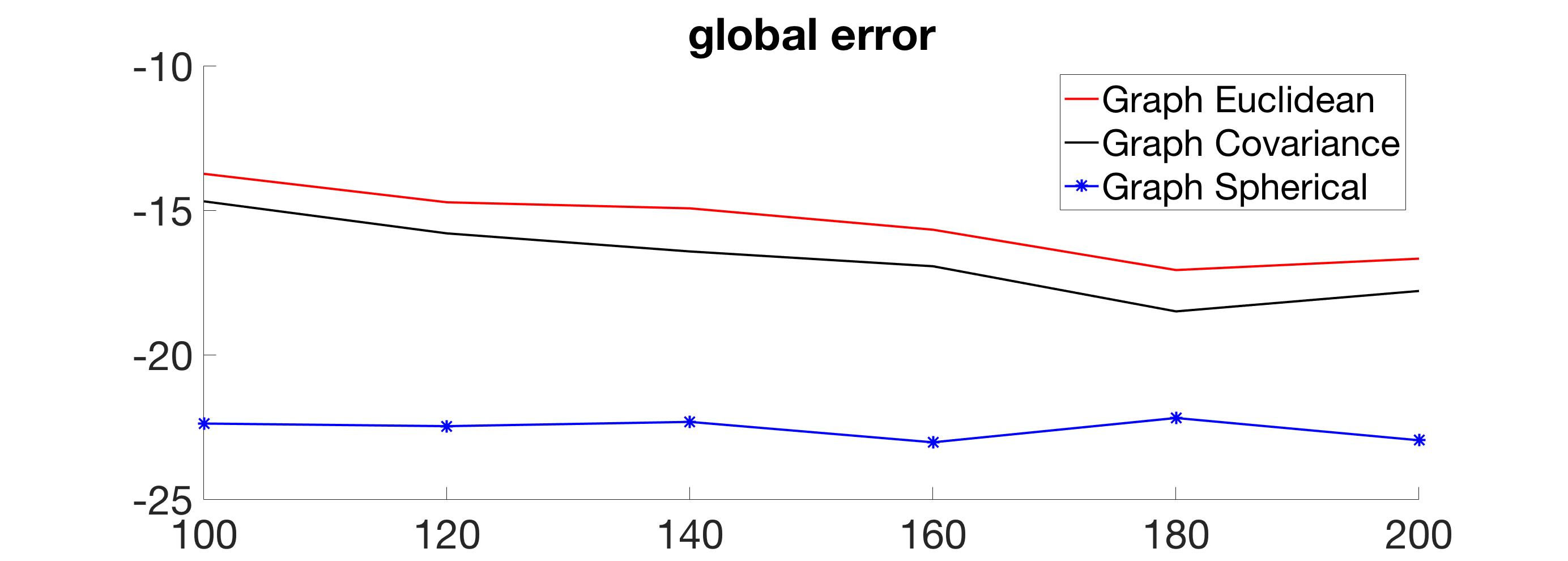}
		\subcaption[sample size vs. log global error plot]{sample size vs. log(global error) plot}\label{fig:2b}
	\end{minipage}%
	\caption{Global error for noisy Euler spiral} \label{fig:NoisyEulerSpiral}
\end{figure}

\vspace{-0.3cm}

Figure \ref{fig:NoisyEulerSpiral} shows that spherical estimation works well on very small data sets, because it efficiently captures the geometry  hidden in the data. 
\subsection{Torus}
We also consider the torus, a two dimensional surface with curvature ranging from negative to positive depending on the location. We set the major radius to be $R = 5$ and the minor radius to be $r=1$ so the equation for the torus is 
$$\left\{x(\theta,\varphi), y(\theta,\varphi),z(\theta,\varphi)\right\}=\left\{(R+r\cos\theta)\cos\varphi,(R+r\cos\theta)\sin\varphi, r\sin\theta\right\}.$$ 

Since the geodesic distance on the torus does not admit an analytic form, we apply the same strategy as in the noisy Euler Spiral case. First we generate a large dataset and apply the Graph Euclidean method to obtain the ``truth'', then estimate the distance through a subset and finally compute the error. Similarly, we also consider the noisy case by adding Gaussian noise to the torus data. The results are shown in Figure \ref{fig:Torus}, which demonstrates that the performance of the spherical estimation is the best for both clean and noisy data.

\begin{figure}[h]
	\centering
	\begin{minipage}{0.35\textwidth}
		\centering
		\includegraphics[width=1\textwidth, height=0.2\textheight]{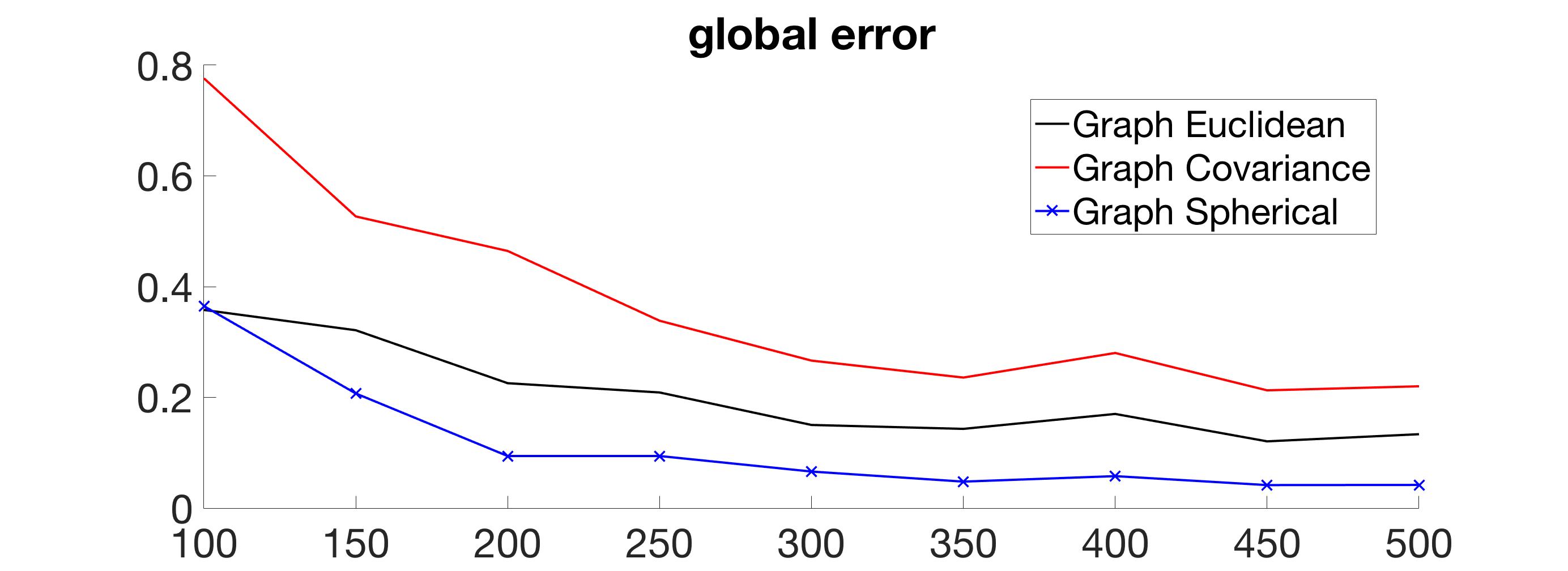}
		\subcaption[Torus]{sample size vs. global error plot for torus}\label{fig:3a}
	\end{minipage}%
	\begin{minipage}{0.25\textwidth}
		\centering
		\includegraphics[width=1\textwidth,height=0.2\textheight]{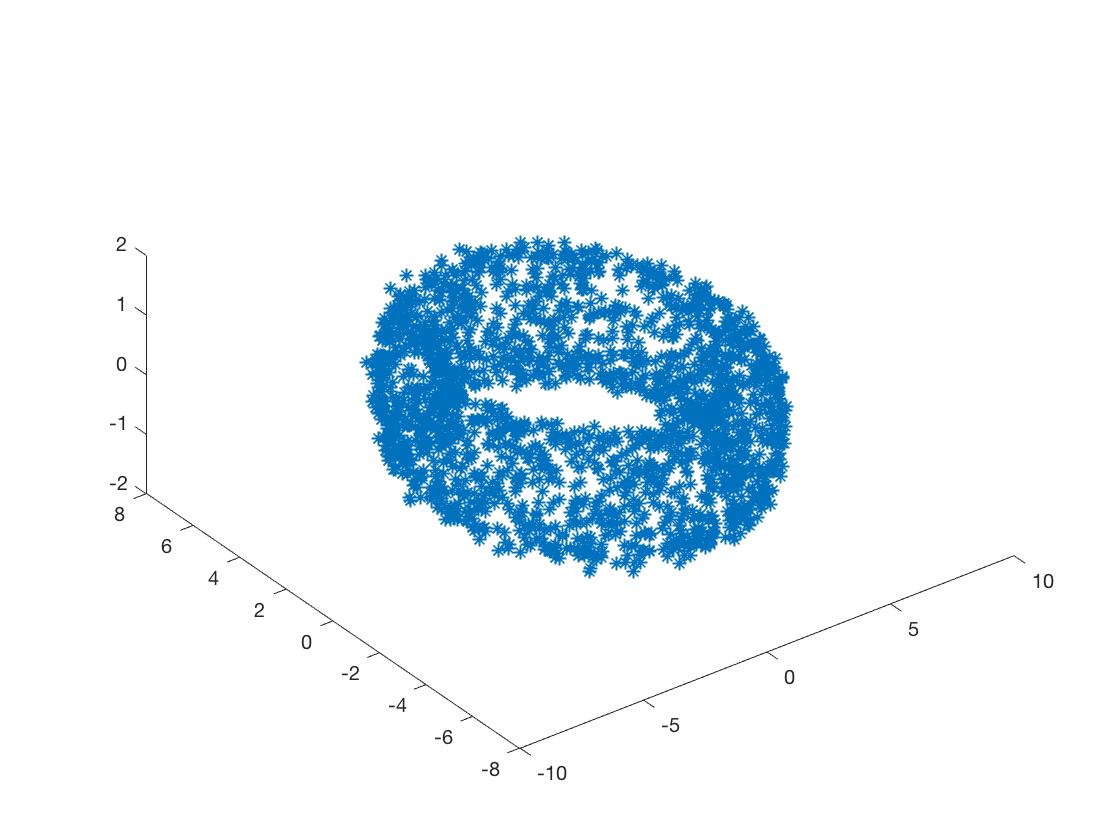}
		\subcaption[Noisy Torus]{Noisy torus}\label{fig:3b}
	\end{minipage}%
	\begin{minipage}{0.35\textwidth}
		\centering
		\includegraphics[width=1\textwidth,height=0.2\textheight]{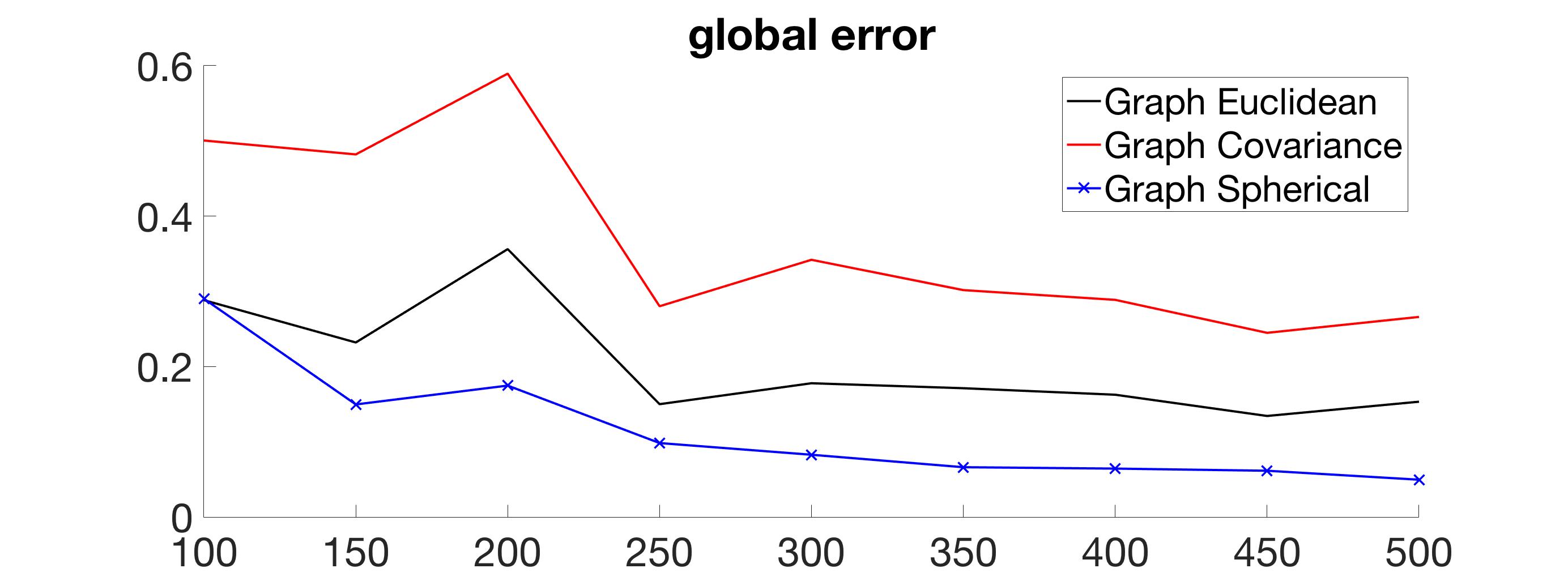}
		\subcaption[sample size vs. log global error plot]{sample size vs. global error  plot for noisy torus}\label{fig:3b}
	\end{minipage}%
	\caption{Global error for (noisy) torus} \label{fig:Torus}
\end{figure}

\section{Applications}
In this section we consider three applications of geodesic distance estimation: clustering, conditional density estimation and regression.
\subsection{$k$-Smedoids clustering}\label{subsec:clustering}
Among the most popular algorithms for clustering, $k$-medoids (introduced in \cite{kaufman1987clustering}) takes the pairwise distance matrix as the input; refer to Algorithm \ref{alg:kmedoids} (\cite{park2009simple}). Similar to $k$-means, $k$-medoids also aims to minimize the distance between the points in each group and the group centers. Differently from $k$-means, the centers are chosen from the data points instead of arbitrary points in the ambient space.  

\begin{algorithm}[!h]
	\caption{$k$-medoids} 
		\label{alg:kmedoids}
	\SetKwData{Left}{left}\SetKwData{This}{this}\SetKwData{Up}{up}
	\SetKwFunction{Union}{Union}\SetKwFunction{FindCompress}{FindCompress}
	\SetKwInOut{Input}{input}\SetKwInOut{Output}{output}
	\Input{Data $\{x_i\}_{i=1}^n\subset \RR^D$, pairwise distance matrix $D\in\RR^{n\times n}$, the number of clusters $K$}
	\Output{Clustered data with labels $\{l_i\}$}
	\BlankLine
	\emph{Initialize: randomly select $K$ points $\{x_{i_j}\}_{j=1}^K\subset \{x_i\}_{i=1}^n$ as medoids}\;
	\emph{Assign $x_i$ label $l_i = \arg\min_{j} D(i_j,i)$}\;
	\emph{Calculate the $cost = \sum_{i} D(i,i_{l_i})$}\;
	\While{the cost decreases}{
		\emph{Swap current medoids and non-medoids}\;
		\emph{Update the label}\;
		\emph{Update the cost}\;
		\If{the cost does not decrease}{
			\emph{Undo the swap}\;
		}
	}
\end{algorithm}

In most packages, the default pairwise distance matrix is the global Euclidean distance $D$, which is inaccurate if the support of the data has essential curvature. As a result, we replace $D$ by $SG$ and call the new algorithm $k$-Smedoids. By estimating $GD$ better, it is reasonable to expect that $k$-Smedoids has better performance.


We present two types of examples: unlabeled (example 1) and labeled data (example 2 and 3). For the unlabeled data, we visualize the clusters to show the performance of different algorithms, for the labeled datasets, we can make use of the labels and do quantitative comparisons. Among clustering performance evaluation metrics, we choose the following: Adjusted Rand Index (ARI, \cite{hubert1985comparing}),  Mutual Information Based Scores (MIBS, \cite{strehl2002cluster}, \cite{vinh2009information}), HOMogeneity (HOM), COMpleteness (COM), V-Measure (VM, \cite{rosenberg2007v}) and Fowlkes-Mallows Scores (FMS, \cite{fowlkes1983method}). We compare these scores for standard $k$-medoids, $k$-Emedoids and our $k$-Smedoids. These algorithms are based on different pairwise distance matrices while other steps are exactly the same, so the performance will illustrate the gain from the estimation of the geodesic distance. We note that for all above metrics, larger values reflect better clustering performance.

Regarding the tuning parameters, depending on the specific problem, $d$ and $k$ can be tuned accordingly. In example 1, the data are visualizable so $d=1$ is known and $k$ can be tuned by the clustering performance: whether the two circles are separated. For example 2-3, cross validation can be applied to tune the parameters based on the six scores. In any case with quantitative scores, cross validation can be used to tune the parameters mentioned above. Our recommended default choices of $k$ are uniformly distributed integers between $d+2$ and $\frac{n}{2}$, proportion to $\sqrt{n}$. Estimating the manifold dimension $d$ has been proven to be a very hard problem, both practically and theoretically. There are some existing methods to estimate $d$, see \cite{granata2016accurate}, \cite{levina2005maximum}, \cite{kegl2003intrinsic}, \cite{camastra2002estimating}, \cite{carter2009local}, \cite{hein2005intrinsic}, \cite{camastra2001intrinsic} and \cite{fan2009intrinsic}, and we can apply these algorithms directly. 
\begin{figure}[h]
	\centering
	\includegraphics[width=0.6\textwidth, height=0.4\textheight]{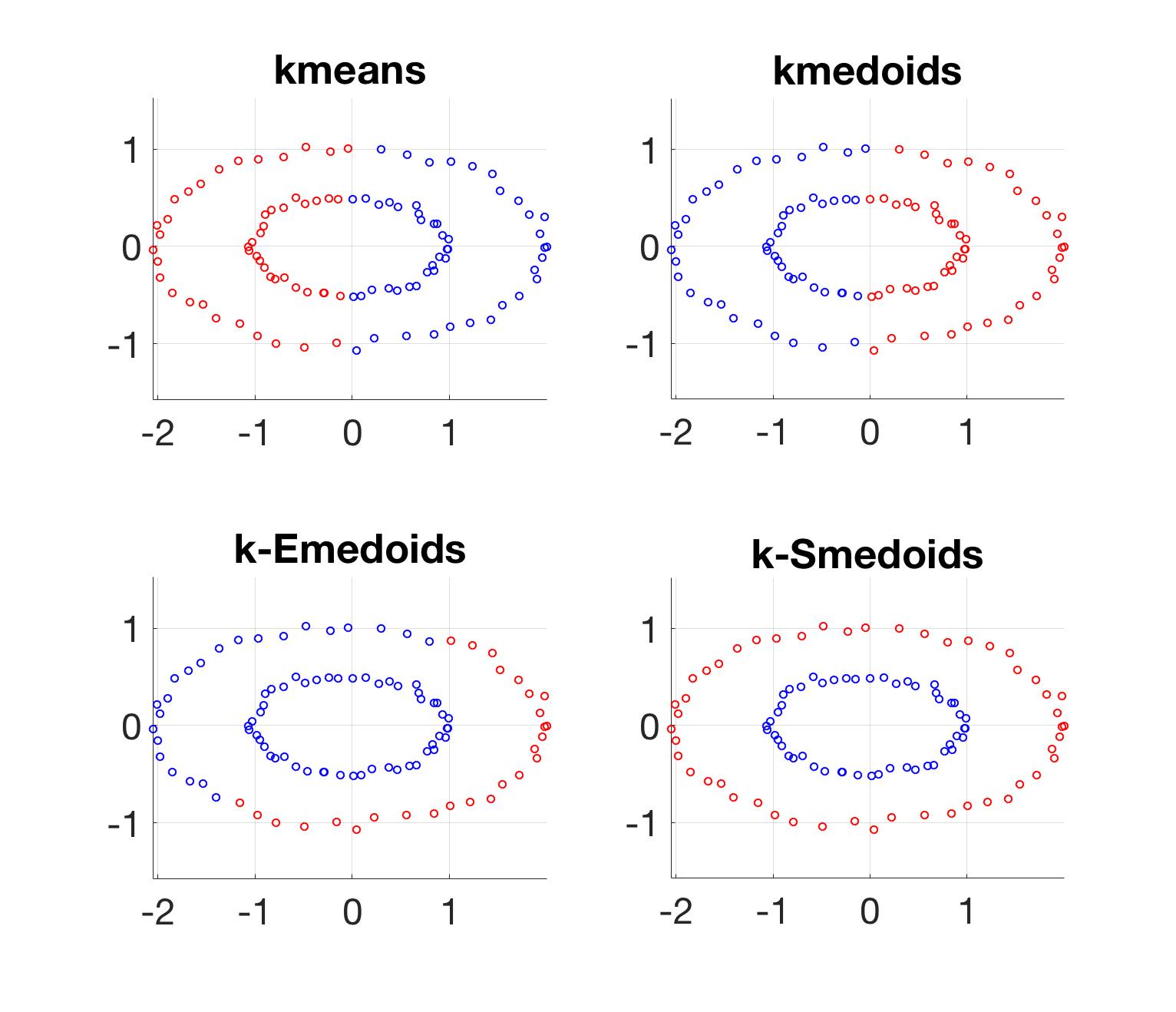}
	\vspace{-0.5cm}
	\caption{Clustering performance for a two ellipses example} \label{fig:twoellipses}
\end{figure}
\noindent{\bf Example 1: Two ellipses.}\label{example1}

We randomly generate $100$ samples from two concentric ellipses with  eccentricity $\sqrt{3}/2$ added by  zero mean Gaussian noise. We compare with $k$-means, standard $k$-medoids and $k$-Emedoids. Figure \ref{fig:twoellipses} shows the clustering results for the two ellipses data. In this example, we set $K=2$,  $k=3$ and $d=1$ since the support is a curve with dimension $1$.

Since the two groups are disconnected and curved, the Euclidean-based algorithms fail while the spherical algorithm works better than using other geodesic distance estimators. We also consider two real datasets with labels.

\noindent{\bf Example 2: Banknote.}

The Banknote data set is introduced in \cite{Banknote}. There are $D=4$ features, characterizing the images from genuine and forged banknote-like specimens and the sample size is $1372$. The binary label indicates whether the banknote specimen is genuine or forged.

Table \ref{table:Banknote} shows the clustering performance of three algorithms for the Banknote data. We can see that $k$-Smedoids has the highest score for all 6 metrics. In this example, $K=2$ is known, and we set $k=4$ and $d=1$. The choice of $d$ and $k$ are determined by cross validation.

\begin{table}[h]
	\begin{minipage}[]{0.5\textwidth}
		\def~{\hphantom{0}}
		\caption {{\small Clustering performance for Banknote}}
		\begin{tabular}{cccc}
			& $k$-medoids& $k$-Emedoids & $k$-Smedoids\\
			ARI & 0.059 &0.004 &{\bf 0.452}\\
			MIBS &  0.041& 0.008 &{\bf0.439}\\
			HOM &0.0416 &0.009& {\bf0.439} \\
			COM & 0.041& 0.138& {\bf0.508}\\
			VM & 0.0415& 0.0163& {\bf0.471}\\
			FMS &0.533& 0.707 &{\bf0.754}
		\end{tabular}
		\label{table:Banknote}
	\end{minipage}
	\begin{minipage}{0.5\textwidth}
		\def~{\hphantom{0}}
		\caption{{\small Clustering performance for Galaxy Zoo}}
		\begin{tabular}{cccc}
			& $k$-medoids& $k$-Emedoids & $k$-Smedoids\\
			ARI & 0.744& 0.805& {\bf0.954}\\
			MIBS &  0.6402 &0.702& {\bf0.900}\\
			HOM &0.712 &0.763& {\bf0.919}\\
			COM & 0.641 & 0.702& {\bf0.900}\\
			VM &  0.674& 0.731 &{\bf0.909}\\
			FMS &0.899 &0.923 &{\bf0.983}
		\end{tabular}
		\label{table:Galaxy}
		
	\end{minipage}
\end{table}

\noindent{\bf Example 3: Galaxy zoo data.}

The last example is from the Galaxy Zoo project available at {\it http://zoo1.galaxyzoo.org}. The features are the fraction of the vote from experts in each of the six categories, and the labels represent whether the galaxy is spiral or elliptical. We randomly choose $1000$ samples from the huge data set. 

Table \ref{table:Galaxy} shows the clustering performance of three algorithms for the Galaxy zoo data. We can see that $k$-Smedoids has the highest score for all 6 metrics. In this example $K=2$, and the parameters $k=6$, $d=1$ are determined by cross validation.   

\subsection{Geodesic conditional density estimation}\label{sec:CDE}
Conditional density estimation aims to estimate $f(y|x)$ based on observations $\{(x_i,y_i)\}_{i=1}^n$ where $x_i\in\RR^D$are predictors and $y_i\in\RR$ are responses. The most popular conditional density estimator involving pairwise distance is the kernel density estimator (KDE) with Gaussian kernel \citep{davis2011remarks}:
\begin{equation}\label{eqn:CKDE}
\hat f(y|x) = \frac{1}{\sqrt{\pi h_2}}\frac{\sum_{i=1}^n e^{-\|x_i-x\|^2/h_1}e^{-(y_i-y)^2/h_2}}{\sum_{i=1}^n e^{-\|x_i-x\|^2/h_1}},
\end{equation}
where $h_1$ and $h_2$ are bandwidths. This method is motivated by the formula $f(y|x) = \frac{f(x,y)}{f(x)}$, using kernel density estimation in both the numerator and the denominator.

As discussed before, if the data have essential curvature, Euclidean distance can't capture the intrinsic structure in the data. Instead, we can improve the performance by replacing the Euclidean distance by geodesic distance. That is, the natural estimator is 

\begin{equation}\label{eqn:GCKDE}
\hat f(y|x) = \frac{1}{\sqrt{\pi h_2}}\frac{\sum_{i=1}^n e^{-d(x_i,x)^2/h_1}e^{-(y_i-y)^2/h_2}}{\sum_{i=1}^n e^{-d(x_i,x)^2/h_1}},
\end{equation}

\noindent where $d$ is the (estimated) geodesic distance. The kernel $e^{-d(x_i,x)^2/h}$ corresponds to the Riemannian Gaussian distribution \citep{said2017riemannian}.

In terms of the distance, we have four pairwise distances between training data: global Euclidean distance $D$, graph Euclidean distance $ID$, graph covariance corrected distance $CD$ and our proposed graph spherical distance $SD$. For any given $x$, $d(x,x_i)$ is obtained by interpolation. First we add $x$ to the graph consists of all training data and connect $x$ with its neighbors. Then we calculate the graph distance between $x$ and $x_i$. This is more efficient than calculating pairwise distances between all samples $X_{train}\cup X_{test}$. The algorithm is formulated in Algorithm \ref{alg:CKDE}:

\begin{algorithm}[!h]
	\caption{Geodesic conditional kernel density estimation algorithm} 
		\label{alg:CKDE}
	\SetKwData{Left}{left}\SetKwData{This}{this}\SetKwData{Up}{up}
	\SetKwFunction{Union}{Union}\SetKwFunction{FindCompress}{FindCompress}
	\SetKwInOut{Input}{input}\SetKwInOut{Output}{output}
	\Input{Training data $\{x_i,y_i\}_{i=1}^n\subset \RR^D\times \RR$, tuning parameters $k$, $d$, $h_1$, $h_2$, given predictor $x$.}
	\Output{Estimated conditional density $\hat f(y|x)$.}
	\BlankLine
	\emph{Estimate pairwise geodesic distance between $x_i$'s $SD$ by Algorithm \ref{alg:global}}\;
	\emph{Calculate $d(x,x_i)$ for neighbors of $x$}\;
	\emph{$d(x,x_i)=d(x,x_{i_0})+SD(i_0,i)$ where $x_{i_0}$ is the closest neighbor of $x$}\;
	\emph{$\hat f(y|x) = \frac{1}{\sqrt{\pi h_2}}\frac{\sum_{i=1}^n e^{-d(x_i,x)^2/h_1}e^{-(y_i-y)^2/h_2}}{\sum_{i=1}^n e^{-d(x_i,x)^2/h_1}}$}.

\end{algorithm}

We can replace the distance estimator in the first step by any other algorithm to obtain the corresponding version of conditional density estimation. To compare the performance, we estimate the conditional density through training data $X_{train}, Y_{train}$ and calculate the sum of log likelihood $\sum_{i=1}^{n_{test}} \log(\hat f(y_i|x_i))$. We randomly permute the data to obtain different training and test sets and provide boxplots for the sum of log likelihoods. 

Regarding the tuning parameters, $k$ and $d$ have been discussed in Section \ref{subsec:clustering}. Regarding bandwidths $h_1$ and $h_2$, there is a very rich literature on choosing optimal bandwidths in other contexts. For simplicity we use cross validation to estimate $h_1$ and $h_2$. We consider the following two real data sets.

\noindent{\bf Example 4: Combined Cycle Power Plant.}
This data set is introduced in \cite{tufekci2014prediction,kaya2012local}, containing $9568$ samples collected from a Combined Cycle Power Plant from 2006 to 2011. There are $4$ predictors: hourly average ambient variables Temperature (T), Ambient Pressure (AP), Relative Humidity (RH) and Exhaust Vacuum (V), to predict the net hourly electrical energy output (EP) of the plant (response). We randomly sample $1000$ data points and repeat $100$ times to obtain the following boxplots for the sum of log likelihood scores for different distance estimation methods, as shown in Figure \ref{fig:PowerPlant}. There is a clear improvement for our graph spherical approach.

\begin{figure}[h]
	\centering
	\includegraphics[width=1\textwidth, height=0.3\textheight]{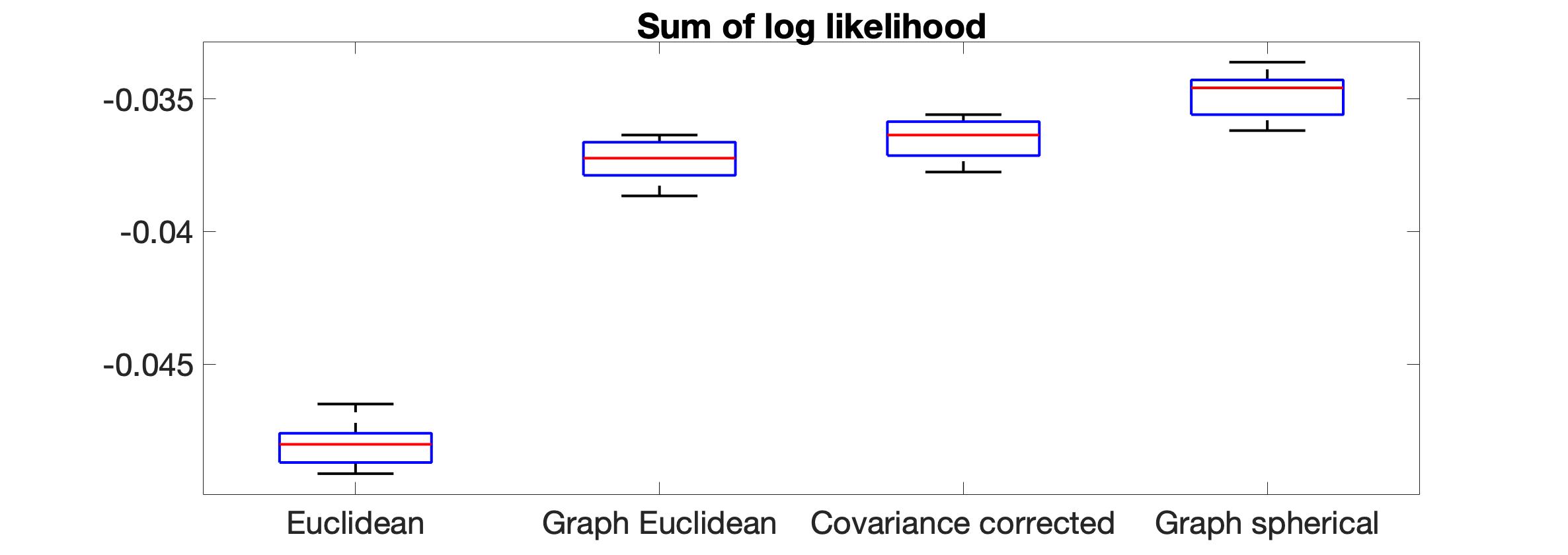}
	\vspace{-0.5cm}
	\caption{Sum of log likelihood for Combined Cycle Power Plant data set} \label{fig:PowerPlant}
\end{figure}

\noindent{\bf Example 5: Concrete Compressive Strength.}
This data set is introduced in \cite{yeh1998modeling}, containing $1030$ samples with $8$ predictors: cement, blast furnace slag, fly ash, water, superplasticizer, coarse aggregate, fine aggregate and age, to predict the concrete compressive strength. We randomly split the data $50-50$ as training and test data and repeat for $100$ times to obtain the boxplots for the sum of log likelihood scores for different distance estimation methods, as shown in Figure \ref{fig:Concrete}.

\begin{figure}[h]
	\centering
	\includegraphics[width=1\textwidth, height=0.3\textheight]{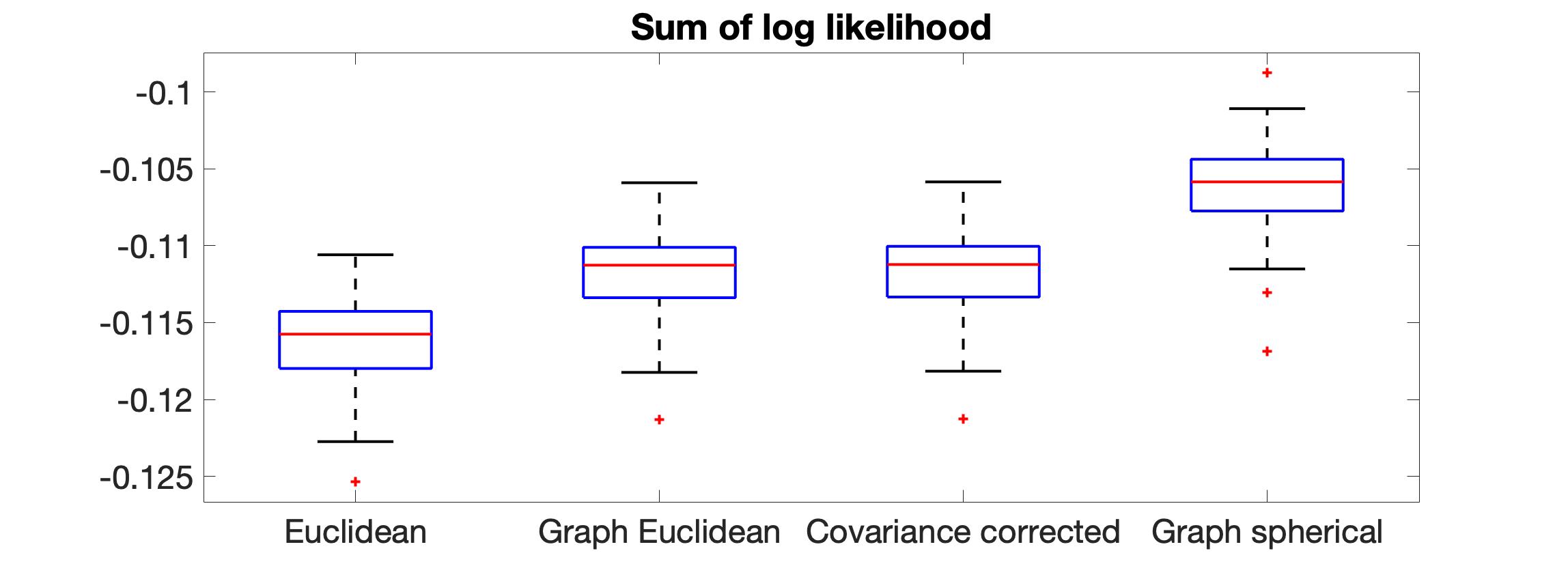}
	\vspace{-0.5cm}
	\caption{Sum of log likelihood for Concrete Compressive Strength data set} 
\label{fig:Concrete}
\end{figure}

From the above two examples we can tell that Euclidean distance is the worst choice because the predictors have non-linear support. Graph Euclidean and covariance corrected distances improve the performance a lot, while graph spherical distance outperforms all competitors.

\subsection{Geodesic kernel mean regression}
As a related case, we also consider kernel mean regression using a simple modification of the Nadaraya-Watson estimator \citep{nadaraya1964estimating,watson1964smooth}:
$$\widehat m(x)=\frac{\sum_{i=1}^n e^{-d(x_i,x)^2/h}y_i}{\sum_{i=1}^n e^{-d(x_i,x)^2/h}}.$$

Algorithm \ref{alg:GKR} provides details:

\begin{algorithm}[!h]
	\caption{Geodesic kernel regression algorithm} 
		\label{alg:GKR}
	\SetKwData{Left}{left}\SetKwData{This}{this}\SetKwData{Up}{up}
	\SetKwFunction{Union}{Union}\SetKwFunction{FindCompress}{FindCompress}
	\SetKwInOut{Input}{input}\SetKwInOut{Output}{output}
	\Input{Training data $\{x_i,y_i\}_{i=1}^n\subset \RR^D\times \RR$, tuning parameters $k$, $d$, $h$, predictor $x$.}
	\Output{Estimated conditional mean $\hat m(x)$.}
	\BlankLine
	\emph{Estimate pairwise geodesic distance between $x_i$'s $SD$ by Algorithm \ref{alg:global}}\;
	\emph{Calculate $d(x,x_i)$ for neighbors of $x$}\;
	\emph{$d(x,x_i)=d(x,x_{i_0})+SD(i_0,i)$ where $x_{i_0}$ is the closest neighbor of $x$}\;
	\emph{$\widehat m(x) = \frac{\sum_{i=1}^n e^{-d(x_i,x)^2/h}y_i}{\sum_{i=1}^n e^{-d(x_i,x)^2/h}}$}.

\end{algorithm}

Again we have four options for the distance in the first step, $D$, $ID$, $CD$ and our proposed $SD$. We measure the performance by calculating the Root Mean Square Error (RMSE) $\sqrt{\sum_{i=1}^{n_{test}}(\widehat m(x_i)-y_i)^2/n_{test}}$, and again use cross validation for bandwidth choice. We consider the same two datasets as in Section \ref{sec:CDE} and show the results in Figure \ref{fig:PowerPlant_GKR} and \ref{fig:Concrete_GKR}:

\begin{figure}[h]
	\centering
	\includegraphics[width=1\textwidth, height=0.3\textheight]{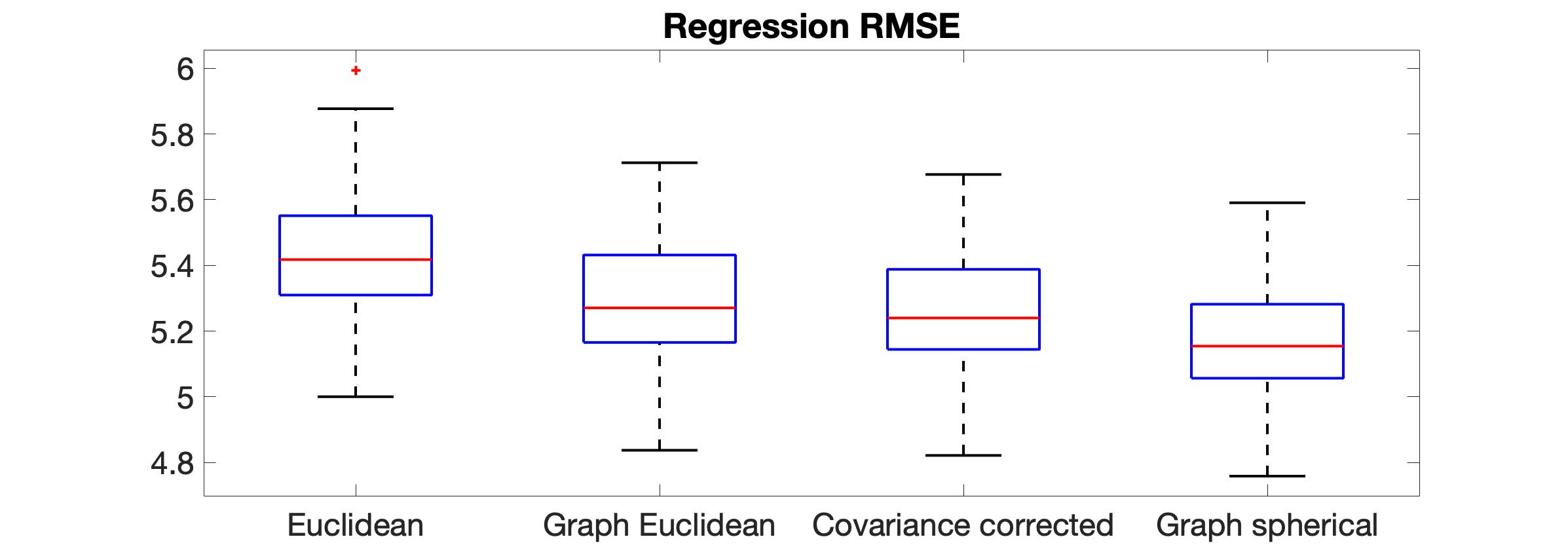}
	\vspace{-0.5cm}
	\caption{Regression RMSE for Combined Cycle Power Plant data set} 
	\label{fig:PowerPlant_GKR}
\end{figure}

\begin{figure}[h]
	\centering
	\includegraphics[width=1\textwidth, height=0.4\textheight]{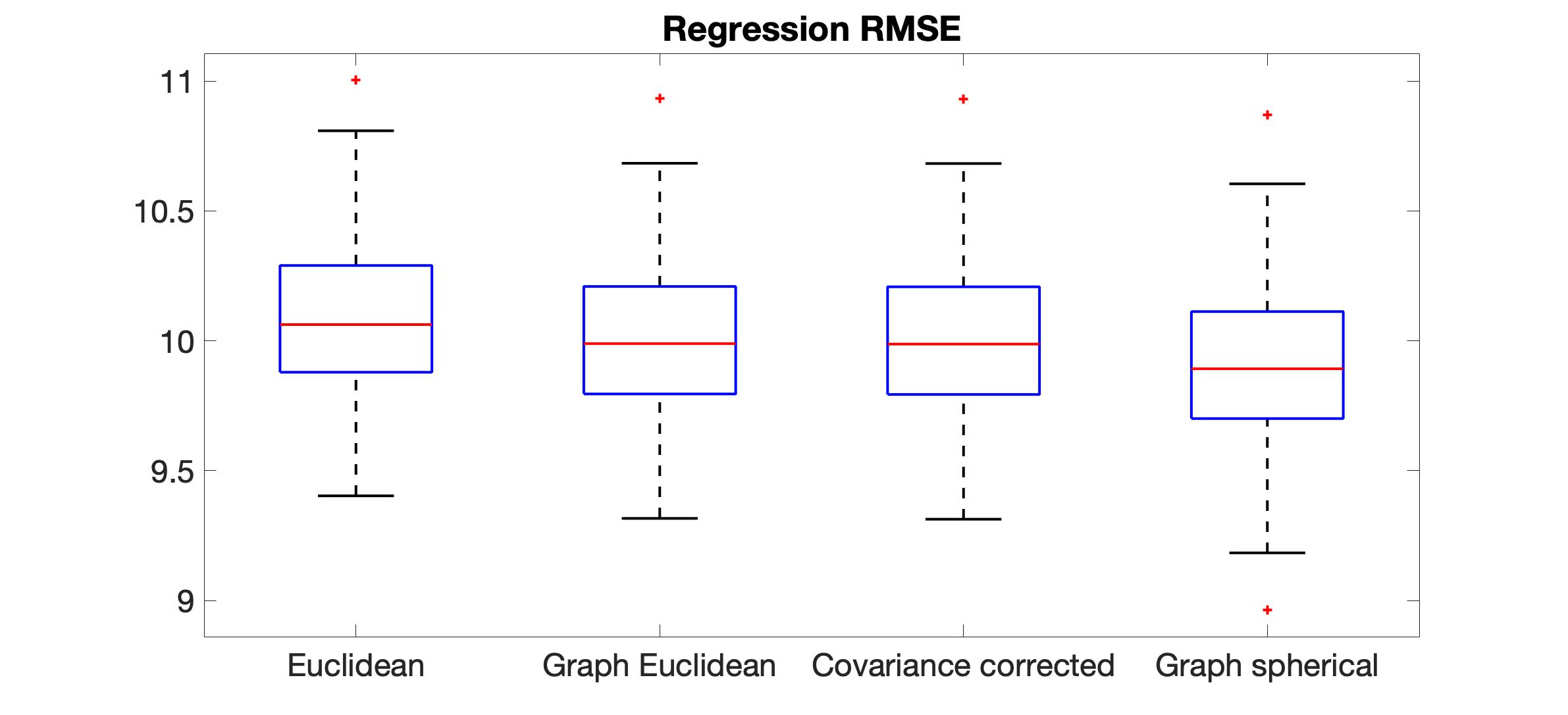}
	\vspace{-0.5cm}
	\caption{Regression RMSE for Concrete Compressive Strength data set} 
	\label{fig:Concrete_GKR}
\end{figure}

Again it is clear that spherical distance outperforms the competitors. 

\section{Discussion}

The choice of distance between data points plays a critical role in many statistical and machine learning procedures.  For continuous data, the Euclidean distance provides by far the most common choice, but we have shown that it can produce badly sub-optimal results in a number of real data examples, as well as for simulated data known to follow a manifold structure up to measurement error.  Our proposed approach seems to provide a clear improvement upon the state-of-the-art for geodesic distance estimation between data points lying close to an unknown manifold, and hence may be useful in many different contexts.

There are multiple future directions of immediate interest.  The first is the question of how the proposed approach performs if the data do not actually have an approximate manifold structure, and whether extensions can be defined that are adaptive to a variety of true intrinsic structures in the data.  We have found that our approach is much more robust to measurement errors than competing approaches; in practice, it is almost never reasonable to suppose that data points fall exactly on an unknown Riemannian manifold.  With this in mind, we allow measurement errors in our approach, so that the data points can deviate from the manifold.  This leads to a great deal of flexibility in practice, and is likely a reason for the good performance we have seen in a variety of real data examples.  However, there is a need for careful work on how to deal with measurement errors and define a single class of distance metrics that can default to Euclidean distance as appropriate or include other structure as appropriate, in an entirely data-dependent manner.

\section*{Acknowledgement}
The authors acknowledge support for this research from an Office of Naval Research grant N00014-14-1-0245/N00014-16-1-2147 and a National Institute of Health grant 5R01ES027498-02.
	
\bibliographystyle{apalike}
\bibliography{ref.bib}

\begin{thebibliography}{}

\bibitem[Bernstein et~al., 2000]{geodist2000}
Bernstein, M., De~Silva, V., Langford, J.~C., and Tenenbaum, J.~B. (2000).
\newblock Graph approximations to geodesics on embedded manifolds.
\newblock Technical report, Technical report, Department of Psychology,
  Stanford University.

\bibitem[Camastra and Vinciarelli, 2001]{camastra2001intrinsic}
Camastra, F. and Vinciarelli, A. (2001).
\newblock Intrinsic dimension estimation of data: An approach based on
  {G}rassberger--{P}rocaccia's algorithm.
\newblock {\em Neural Processing Letters}, 14(1):27--34.

\bibitem[Camastra and Vinciarelli, 2002]{camastra2002estimating}
Camastra, F. and Vinciarelli, A. (2002).
\newblock Estimating the intrinsic dimension of data with a fractal-based
  method.
\newblock {\em IEEE Transactions on Pattern Analysis and Machine Intelligence},
  24(10):1404--1407.

\bibitem[Carter et~al., 2009]{carter2009local}
Carter, K.~M., Raich, R., and Hero~III, A.~O. (2009).
\newblock On local intrinsic dimension estimation and its applications.
\newblock {\em IEEE Transactions on Signal Processing}, 58(2):650--663.

\bibitem[Coope, 1993]{coope1993circle}
Coope, I.~D. (1993).
\newblock Circle fitting by linear and nonlinear least squares.
\newblock {\em Journal of Optimization Theory and Applications},
  76(2):381--388.

\bibitem[Davis et~al., 2011]{davis2011remarks}
Davis, R.~A., Lii, K.-S., and Politis, D.~N. (2011).
\newblock Remarks on some nonparametric estimates of a density function.
\newblock In {\em Selected Works of Murray Rosenblatt}, pages 95--100.
  Springer.

\bibitem[Dijkstra, 1959]{dijkstra1959note}
Dijkstra, E.~W. (1959).
\newblock A note on two problems in connexion with graphs.
\newblock {\em Numerische Mathematik}, 1(1):269--271.

\bibitem[Fan et~al., 2009]{fan2009intrinsic}
Fan, M., Qiao, H., and Zhang, B. (2009).
\newblock Intrinsic dimension estimation of manifolds by incising balls.
\newblock {\em Pattern Recognition}, 42(5):780--787.

\bibitem[Floyd, 1962]{floyd1962algorithm}
Floyd, R.~W. (1962).
\newblock Algorithm 97: shortest path.
\newblock {\em Communications of the ACM}, 5(6):345.

\bibitem[Fowlkes and Mallows, 1983]{fowlkes1983method}
Fowlkes, E.~B. and Mallows, C.~L. (1983).
\newblock A method for comparing two hierarchical clusterings.
\newblock {\em Journal of the American Statistical Association},
  78(383):553--569.

\bibitem[Granata and Carnevale, 2016]{granata2016accurate}
Granata, D. and Carnevale, V. (2016).
\newblock Accurate estimation of the intrinsic dimension using graph distances:
  Unraveling the geometric complexity of datasets.
\newblock {\em Scientific Reports}, 6:31377.

\bibitem[Hein and Audibert, 2005]{hein2005intrinsic}
Hein, M. and Audibert, J.-Y. (2005).
\newblock Intrinsic dimensionality estimation of submanifolds in r d.
\newblock In {\em Proceedings of the 22nd International Conference on Machine
  learning}, pages 289--296. ACM.

\bibitem[Hubert and Arabie, 1985]{hubert1985comparing}
Hubert, L. and Arabie, P. (1985).
\newblock Comparing partitions.
\newblock {\em Journal of Classification}, 2(1):193--218.

\bibitem[Kaufman et~al., 1987]{kaufman1987clustering}
Kaufman, L., Rousseeuw, P., and Dodge, Y. (1987).
\newblock Clustering by means of medoids in statistical data analysis based on
  the l1 norm and related method.
\newblock {\em North-Holland}, pages 405--416.

\bibitem[Kaya et~al., 2012]{kaya2012local}
Kaya, H., T{\"u}fekci, P., and G{\"u}rgen, F.~S. (2012).
\newblock Local and global learning methods for predicting power of a combined
  gas \& steam turbine.
\newblock In {\em Proceedings of the International Conference on Emerging
  Trends in Computer and Electronics Engineering (ICETCEE)}, pages 13--18.

\bibitem[K{\'e}gl, 2003]{kegl2003intrinsic}
K{\'e}gl, B. (2003).
\newblock Intrinsic dimension estimation using packing numbers.
\newblock In {\em Advances in Neural Information Processing Systems}, pages
  697--704.

\bibitem[Levina and Bickel, 2005]{levina2005maximum}
Levina, E. and Bickel, P.~J. (2005).
\newblock Maximum likelihood estimation of intrinsic dimension.
\newblock In {\em Advances in Neural Information Processing Systems}, pages
  777--784.

\bibitem[Li et~al., 2018]{spherelets}
Li, D., Mukhopadhyay, M., and Dunson, D.~B. (2018).
\newblock Efficient manifold and subspace approximations with spherelets.
\newblock {\em arXiv preprint arXiv:1706.08263}.

\bibitem[Lohweg and Doerksen, 2012]{Banknote}
Lohweg, V. and Doerksen, H. (2012).
\newblock {UCI} machine learning repository.

\bibitem[Malik et~al., 2019]{malik2019connecting}
Malik, J., Shen, C., Wu, H.-T., and Wu, N. (2019).
\newblock Connecting dots: from local covariance to empirical intrinsic
  geometry and locally linear embedding.
\newblock {\em Pure and Applied Analysis}, 1(4):515--542.

\bibitem[Meng et~al., 2008]{geodist2008}
Meng, D., Leung, Y., Xu, Z., Fung, T., and Zhang, Q. (2008).
\newblock Improving geodesic distance estimation based on locally linear
  assumption.
\newblock {\em Pattern Recognition Letters}, 29(7):862--870.

\bibitem[Meng et~al., 2007]{geodist2007}
Meng, D., Xu, Z., Gu, N., and Dai, M. (2007).
\newblock Estimating geodesic distances on locally linear patches.
\newblock In {\em Signal Processing and Information Technology, 2007 IEEE
  International Symposium on}, pages 851--854. IEEE.

\bibitem[Nadaraya, 1964]{nadaraya1964estimating}
Nadaraya, E.~A. (1964).
\newblock On estimating regression.
\newblock {\em Theory of Probability \& Its Applications}, 9(1):141--142.

\bibitem[Park and Jun, 2009]{park2009simple}
Park, H.-S. and Jun, C.-H. (2009).
\newblock A simple and fast algorithm for k-medoids clustering.
\newblock {\em Expert Systems with Applications}, 36(2):3336--3341.

\bibitem[Rosenberg and Hirschberg, 2007]{rosenberg2007v}
Rosenberg, A. and Hirschberg, J. (2007).
\newblock V-measure: A conditional entropy-based external cluster evaluation
  measure.
\newblock In {\em Proceedings of the 2007 Joint Conference on Empirical Methods
  in Natural Language Processing and Computational Natural Language Learning
  (EMNLP-CoNLL)}.

\bibitem[Said et~al., 2017]{said2017riemannian}
Said, S., Bombrun, L., Berthoumieu, Y., and Manton, J.~H. (2017).
\newblock Riemannian gaussian distributions on the space of symmetric positive
  definite matrices.
\newblock {\em IEEE Transactions on Information Theory}, 63(4):2153--2170.

\bibitem[Silva and Tenenbaum, 2003]{silva2003global}
Silva, V.~D. and Tenenbaum, J.~B. (2003).
\newblock Global versus local methods in nonlinear dimensionality reduction.
\newblock In {\em Advances in Neural Information Processing Systems}, pages
  721--728.

\bibitem[Smolyanov et~al., 2007]{smolyanov2007chernoff}
Smolyanov, O.~G., Weizs{\"a}cker, H.~v., and Wittich, O. (2007).
\newblock Chernoff's theorem and discrete time approximations of brownian
  motion on manifolds.
\newblock {\em Potential Analysis}, 26(1):1--29.

\bibitem[Strehl and Ghosh, 2002]{strehl2002cluster}
Strehl, A. and Ghosh, J. (2002).
\newblock Cluster ensembles---a knowledge reuse framework for combining
  multiple partitions.
\newblock {\em Journal of Machine Learning Research}, 3(Dec):583--617.

\bibitem[Tenenbaum et~al., 2000]{isomap2000}
Tenenbaum, J.~B., De~Silva, V., and Langford, J.~C. (2000).
\newblock A global geometric framework for nonlinear dimensionality reduction.
\newblock {\em Science}, 290(5500):2319--2323.

\bibitem[T{\"u}fekci, 2014]{tufekci2014prediction}
T{\"u}fekci, P. (2014).
\newblock Prediction of full load electrical power output of a base load
  operated combined cycle power plant using machine learning methods.
\newblock {\em International Journal of Electrical Power \& Energy Systems},
  60:126--140.

\bibitem[Vinh et~al., 2009]{vinh2009information}
Vinh, N.~X., Epps, J., and Bailey, J. (2009).
\newblock Information theoretic measures for clusterings comparison: is a
  correction for chance necessary?
\newblock In {\em Proceedings of the 26th Annual International Conference on
  Machine Learning}, pages 1073--1080. ACM.

\bibitem[Warshall, 1962]{warshall1962theorem}
Warshall, S. (1962).
\newblock A theorem on boolean matrices.
\newblock In {\em Journal of the ACM}. Citeseer.

\bibitem[Watson, 1964]{watson1964smooth}
Watson, G.~S. (1964).
\newblock Smooth regression analysis.
\newblock {\em Sankhy{\=a}: The Indian Journal of Statistics, Series A}, pages
  359--372.

\bibitem[Wu et~al., 2018]{wu2018think}
Wu, H.-T., Wu, N., et~al. (2018).
\newblock Think globally, fit locally under the manifold setup: Asymptotic
  analysis of locally linear embedding.
\newblock {\em The Annals of Statistics}, 46(6B):3805--3837.

\bibitem[Yang, 2004]{geodist2004}
Yang, L. (2004).
\newblock K-edge connected neighborhood graph for geodesic distance estimation
  and nonlinear data projection.
\newblock In {\em Pattern Recognition, 2004. ICPR 2004. Proceedings of the 17th
  International Conference on}, volume~1, pages 196--199. IEEE.

\bibitem[Yeh, 1998]{yeh1998modeling}
Yeh, I.-C. (1998).
\newblock Modeling of strength of high-performance concrete using artificial
  neural networks.
\newblock {\em Cement and Concrete Research}, 28(12):1797--1808.

\end{thebibliography}

	\section{Appendix}
	\subsection{Centered $k$-osculating sphere solution}
	We solve the optimization problem described in Definition \ref{def:CSPCA}.
	\begin{proof}[Proof of Theorem \ref{thm:CSPCA}]
		Our goal is to minimize the function
		$$f(c)=\sum_{x_j\in X^{[k]}_i}\left(\|y_j-c\|^2-\|y_i-c\|^2\right)^2.$$
		First we simply $f(c)$ and omit the terms that do not depend on $c$ in the third equation:
		\begin{align*}
		f(c)& = \sum_{x_j\in X^{[k]}_i}\left\{(y_j-c)^\top (y_j-c) -(y_i-c)^\top (x_i-c)\right\}^2\\
		& =\sum_{x_j\in X^{[k]}_i}\left\{\|y_j\|^2-\|y_i\|^2-2c^\top(y_j-y_i)\right\}^2\\
		& = 4c^\top  \sum_{x_j\in X^{[k]}_i}\left\{ (y_j-y_i)(y_j-y_i)^\top\right\} c-4c^\top\sum_{x_j\in X^{[k]}_i}(\|y_j\|^2-\|y_i\|^2)(y_j-y_i)\\
		& = 4c^\top H c-4c^\top f,
		\end{align*}	
		where $H =\sum_{x_j\in X^{[k]}_i}\left\{ (y_j-y_i)(y_j-y_i)^\top\right\}$ and $f=\sum_{x_j\in X^{[k]}_i}(\|y_j\|^2-\|y_i\|^2)(y_j-y_i)$.
		
		This is a quadratic function with respect to $c$ and the minimizer is given by
		$$c^* = \frac{1}{2}H^{-1}f.$$
		Then $r^* = \|x_i-c^*\|$ is the radius of the centered $k$-osculating sphere.\end{proof}
	
	\subsection{Local estimation error}
	In this section we prove Theorem \ref{thm:curve}-\ref{thm:general}. 
	\subsubsection{Proof for curves}
	First we prove Theorem \ref{thm:curve}, the error bound for curves.
	\begin{proof}[Proof of Theorem \ref{thm:curve}]
		At the fixed points $x=\gamma(0)$, let $\boldsymbol{t}=\gamma'(0)$ and $\bold{n}=\frac{\gamma''(0)}{\|\gamma''(0)\|}$, then $(-\bold{n},\boldsymbol{t})$ is an orthonormal basis at $x$. Then the Taylor expansion $\gamma(s)=\gamma(0)+\gamma'(0)s+\frac{s^2}{2}\gamma''(0)+\frac{s^3}{6}\gamma'''(s)+O(s^4)$ can be written in the new coordinates as
		$$\gamma(s)=\bold{0_2}+\begin{bmatrix}
		0\\s
		\end{bmatrix}+\begin{bmatrix}
		-\frac{\kappa}{2}s^2\\0
		\end{bmatrix}+\begin{bmatrix}
		v_1 s^3\\v_2 s^3
		\end{bmatrix}+O(s^4)=\begin{bmatrix}
		-\frac{\kappa}{2}s^2+v_1s^3\\s+v_2s^3
		\end{bmatrix}+O(s^4),$$
		where $v_1$ and $v_2$ are unknown constants subject to the constraint $\|\gamma'(s)\|=1$. Observe that $\gamma'(s)=\begin{bmatrix}
		-\kappa s +3v_1 s^2\\
		1+3v_2 s^2
		\end{bmatrix}+O(s^3)$ so $\|\gamma(s)\|^2=1+\kappa^2s^2+6v_2s^2+O(s^3)=1$ then we conclude that $v_2=-\frac{\kappa^2}{6}$. As a result, $\gamma(s)=\begin{bmatrix}
		-\frac{\kappa}{2}s^2+v_1s^3\\s-\frac{\kappa^2}{6}s^3
		\end{bmatrix}+O(s^4).$ For convenience, we assume $\kappa>0$; the proof is the same when $\kappa\leq 0$.

		Let $\theta$ be the angle between $x-c$ and $y-c$. Then the spherical distance between $\pi(y)$ and $x$ is $r\theta$. In section \ref{local_dist}, we characterize $\theta$ by $\cos(\theta)=r \arccos\left\{\frac{x-c}{r}\cdot \frac{\pi(y)-c}{r}\right\}$, here we characterize $\theta$ by $\tan(\theta)$ for computational simplicity. Let $y_t$ be the intersection of the tangent space spanned by $\gamma'(0)$ and the straight line connecting $y$ and $c$, and let $y_l$ be the projection of $y$ onto the tangent space. Then we can focus on the triangle $\triangle cxy_l$ with $x-c\parallelsum y_l-y$. Observe that
		$$\tan(\theta)= \frac{\|y_t-x\|}{\|x-c\|}=\frac{\|y_t-y_l\|}{\|y-y_l\|}=\frac{\|y_t-x\|-\|y_l-x\|}{\|y-y_l\|},$$
		so $\|y_t-x\|\|y-y_l\|=r\|y_t-x\|-r\|y_l-x\|$, hence $(r-\|y-y_l\|)\|y_t-x\|=r\|y_l-x\|$. As a result,
		\begin{equation}\label{tan}
		\tan(\theta)=\frac{\|y_t-x\|}{r}=\frac{\|y_l-x\|}{r-\|y-y_l\|}.
		\end{equation}
		By the definition of $y_l$, we can write $y_l=\begin{bmatrix}
		0\\s-\frac{\kappa^2}{6}s^3+O(s^4)
		\end{bmatrix}$ and similarly $y-y_l = \begin{bmatrix}
		\frac{\kappa}{2}s^2-v_1s^3+O(s^4)\\0
		\end{bmatrix}$. Plugging these coordinates in equation \ref{tan}, we have
		\begin{align}
		\tan(\theta) & = \frac{\|y_l-x\|}{r-\|y-y_l\|}= \frac{s-\frac{\kappa^2}{6}s^3+O(s^4)}{r-\left(\frac{\kappa}{2}s^2-v_1s^3+O(s^4)\right)} \nonumber\\
		&  =\frac{1}{r}\left\{s-\frac{\kappa^2}{6}s^3+O(s^4)\right\}\left[1-\left\{\frac{\kappa}{2r}s^2-\frac{v_1}{r}s^3+O(s^4)\right\}\right]^{-1} \nonumber\\
		& = \frac{1}{r}\left\{s-\frac{\kappa^2}{6}s^3+O(s^4)\right\}\left\{1+\frac{\kappa}{2r}s^2-\frac{v_1}{r}s^3+O(s^4) \right\}\nonumber\\
		& = \frac{1}{r}\left\{s+\frac{\kappa}{2r}s^3-\frac{\kappa^2}{6}s^3+O(s^4)\right\}. \label{eqn:tan}
		\end{align}
		Finally, by the Taylor expansion of $\arctan$, the spherical distance is 
		\begin{align}
		d_S(x,y)&=r\theta = r\arctan\left\{\tan(\theta)\right\}=r\arctan\left[\frac{1}{r}\left\{s+\frac{\kappa}{2r}s^3-\frac{\kappa^2}{6}s^3+O(s^4)\right\}\right]\nonumber\\
		& = r\left[\frac{1}{r}\left\{s+\frac{\kappa}{2r}s^3-\frac{\kappa^2}{6}s^3+O(s^4)\right\}-\frac{1}{3r^3}s^3+O(s^4)\right]\nonumber\\
		& = s+\left(\frac{\kappa}{2r}-\frac{\kappa^2}{6}-\frac{1}{3r^2}\right)+O(s^4) \nonumber\\
		& = s-\frac{1}{6}\left(\kappa-\frac{1}{r}\right)\left(\kappa-\frac{2}{r}\right)s^3+O(s^4)=s+O(s^4).\label{eqn:dist}
		\end{align}
		The last step results from the fact that $\kappa=\frac{1}{r}$.
	\end{proof}
	
	\subsubsection{Proof for hyper-surfaces}
	Based on the proof of Theorem \ref{thm:curve}, we next prove Theorem \ref{thm:hypersurf}, the error bound for hyper-surfaces.
	
	\begin{proof}[Proof of Theorem \ref{thm:hypersurf}]
		Define $\gamma(s)\coloneqq \exp_x(sv)$ and denote the normal vector of $T_x M$ by $\boldsymbol{n}$. Let $v=\gamma'(0)$ and expand $v$ to an orthonormal  basis of $T_xM$, denoted by $\{v,v_2,\cdots,v_d\}$. Since $\boldsymbol{n}\perp T_xM$, $\{-n,v,v_2,\cdots,v_d\}$ forms a basis for $\RR^{d+1}$. Before doing Taylor expansion, we first prove that $\gamma''(0)\parallelsum \boldsymbol{n}$, where $\parallelsum$ representing parallel relation. Denote the projection to $\boldsymbol{n}$ by $^\perp$ and the covariant derivative by $\nabla$. Since $\gamma$ is a geodesic, we have
		$$0=\nabla_\gamma'\gamma'=\gamma''-(\gamma'')^\perp,$$
		which implies $\gamma''(0)\parallelsum\boldsymbol{n}$. By Taylor expansion, we rewrite $\gamma$ in terms of the new coordinates:
		$$\gamma(s)=\bold{0_{d+1}}+\begin{bmatrix}
		0\\s\\0\\ \vdots\\0
		\end{bmatrix}+\begin{bmatrix}
		-\frac{\kappa_y}{2}s^2\\0\\\vdots\\0
		\end{bmatrix}+\begin{bmatrix}
		\alpha_1 s^3\\\alpha_2 s^3\\\vdots\\\alpha_{d+1} s^3
		\end{bmatrix}+O(s^4)=\begin{bmatrix}
		-\frac{\kappa}{2}s^2+\alpha_1s^3\\s+\alpha_2s^3\\ \alpha_3 s^3\\\vdots\\\alpha_{d+1}s^3
		\end{bmatrix}+O(s^4).$$
		Again, by the constraint $\|\gamma'\|=1$, we have $\alpha_2= -\frac{\kappa_y^2}{6}$. As before, denote the intersection of $y-c$ and $T_xM$ by $y_t$ and the angle between $x-c$ and $y-c$ by $\theta$, so $$\tan(\theta)=\frac{\|y_t-x\|}{\|x-c\|}=\frac{\|y_t-x\|}{r}.$$
		By direct calculation, we derive that the coordinates for $y_t$ as
		$$y_t=\begin{bmatrix}
		0\\ss_t-\frac{\kappa_y^2}{6}s^3s_t\\ O(s^3)s_t\vdots\\ O(s^3)s_t
		\end{bmatrix}+O(s^4),$$
		where $s_t = \frac{1}{1+\frac{\kappa_y}{2r}s^2+O(s^3)}=1+\frac{\kappa_y}{2r}s^2+O(s^3)$. As a result,
		\begin{align*}
		\tan(\theta)&=\frac{\|y_t-x\|}{r}=\frac{1}{r}\left\{ s^2s_t^2-\frac{\kappa_y^2}{3}s^4s_t^2 +O(s^6)  \right\}^{\frac{1}{2}}\\
		& = \frac{s}{r}\left[ 1+\frac{\kappa_y}{r}s^2+O(s^3) -\frac{\kappa_y^2}{3}s^2\left\{1+\frac{\kappa_y}{r}s^2+O(s^3)\right\}  \right]^{\frac{1}{2}}\\
		&  = \frac{s}{r}\left\{ 1+\frac{\kappa_y}{r}s^2-\frac{\kappa_y^2}{3}s^2 +O(s^3)  \right\}^{\frac{1}{2}}= \frac{s}{r}\left\{ 1+\frac{\kappa_y}{2r}s^2-\frac{\kappa_y^2}{6}s^2 +O(s^3)  \right\}\\
		& = \frac{1}{r}\left\{ s+\frac{\kappa_y}{2r}s^3-\frac{\kappa_y^2}{6}s^3 +O(s^4)  \right\}.
		\end{align*}
		This is exactly the same as Equation \ref{eqn:tan}. Then similar to the proof of Equation \ref{eqn:dist}, we conclude that 
		\begin{align}
		d_S(x,y) & = r \theta = s-\frac{1}{6}\left(\kappa_y-\frac{1}{r}\right)\left(\kappa_y-\frac{2}{r}\right)s^3+O(s^4)\nonumber\\
		& =s+(\kappa_y-\kappa_0)(\kappa_y-2\kappa_0)s^3+O(s^4).\label{eqn:sph_dist}
		\end{align}\end{proof}
	Now we can prove Corollary \ref{cly:hypersurf}.
	\begin{proof}[Proof of Corollary \ref{cly:hypersurf}]
		\begin{enumerate}[(1)]
			\item From Equation \ref{eqn:sph_dist}, the third order term vanishes if and only if $\kappa_y=\kappa_0$ or $\kappa_y=2\kappa_0$.
			\item Again by  Equation \ref{eqn:sph_dist}, we have
			\begin{align*}
			|d_S(x,y) -s|\leq& = |\kappa_y-\kappa_0||\kappa_y-\kappa_0-\kappa_0|s^3=O(s^4)\\
			&\leq \bar r (\bar r + |\kappa_0|)s^3+O(s^4)\leq O(\bar r s^3).
			\end{align*}
		\end{enumerate}
	\end{proof}
	\subsubsection{Proof for general cases}
	Finally we prove Theorem \ref{thm:general}.
	\begin{proof}[Proof of Theorem \ref{thm:general}]
		For $y=\exp_x(sv)=\gamma(s)$, let $v_1=v$, and expand $v_1$ to an orthonormal basis of $T_xM$ denoted by $\{v_1,\cdots,v_d\}$. Similarly, fix any orthonormal basis of the normal space $T_xM^\perp$, denote by $\{\boldsymbol{n_1},\cdots,\boldsymbol{n_{D-d}}\}$. The same as the proof of Theorem \ref{thm:hypersurf}, $\gamma''(0)\in T_xM^\perp$, so the Taylor expansion of $\gamma$ can be expressed in the new coordinates as follows:
		$$\gamma(s)=\bold{0_D}+\begin{bmatrix}
		\boldsymbol{0_D-d}\\s\\\boldsymbol{0_{d-1}}
		\end{bmatrix}+\begin{bmatrix}
		\frac{\alpha_1}{2}s^2\\\vdots\\\frac{\alpha_{D-d}}{2}s^2\\\boldsymbol{0_d}
		\end{bmatrix}+O(s^3)=\begin{bmatrix}
		\frac{\alpha_1}{2}s^2\\\vdots\\ \frac{\alpha_{D-d}}{2} s^2\\s\\\boldsymbol{0_{d-1}}
		\end{bmatrix}+O(s^3),$$
		where $\gamma''(0)=[\alpha_1,\cdots,\alpha_{D-d},\boldsymbol{0_d}]^\top$. 
		
		Assume $c = x+r\boldsymbol{n}$ where $\bold{n}=\sum_{i=1}^{D-d}\beta_i \boldsymbol{n_i}$ is a unit vector in the normal space $T_xM^\perp$ so $\sum_{i=1}^{D-d}\beta_i^2=1$. The idea of the proof is to connect $s$ and $d_S(x,y)$ by the tangent space. To be more specific, denote the projection onto $T_xM $ by $P_x$, and define $y_l = P_x(y)$ and $y_s=P_x\left\{\pi(y)\right\}$ . Then it suffices to show the following three statements:
		\begin{enumerate}[i]
			\item  $d_M(x,y)= \|y_l-x\| + O(s^3)$.\label{stmt:general}
			\item $d_S(x,y) = \|y_s-x\|+O(s^3)$.\label{stmt:sphere}
			\item $ \|y_s-x\|=s+O(s^3)$. \label{stmt:euc}
		\end{enumerate} 
		Observe that the first statement implies that the Euclidean distance between base point and the projection to the tangent space is an estimator of the geodesic distance with error $O(s^3)$. Since $T_xM$ is the common tangent space of $M$ and $S_x(c,r)$, Statement \ref{stmt:general} implies Statement \ref{stmt:sphere}.
		So it suffices to show Statement \ref{stmt:general} and Statement \ref{stmt:euc}.
		Before we prove the statements, we need to calculate the coordinates of $y_l$ and $y_s$. By the definition of $P_x$, we have
		\begin{equation}\label{eqn:y_l}
		y_l = P_x(y)=\begin{bmatrix}
		\boldsymbol{0_{D-d}}\\s+O(s^3)\\O(s^3)
		\end{bmatrix}.
		\end{equation}
		
		Similarly, by the definition of $\pi$ and linearity of $P_x$, 
		\begin{align}
		y_s &=P_x\left\{c+\frac{r(y-c)}{\|y-c\|}\right\}=P_x(c)+\frac{r}{\|y-c\|}P_x(y)-\frac{r}{\|y-c\|}P_x(c)\nonumber\\
		& = \boldsymbol{0_D}+\frac{r}{\|y-c\|}
		\begin{bmatrix}
		\boldsymbol{0_{D-d}}\\s+O(s^3)\\ \boldsymbol{0_{d-1}}+O(s^3)
		\end{bmatrix}-\boldsymbol{0_D}=\frac{r}{\|y-c\|}
		\begin{bmatrix}
		\boldsymbol{0_{D-d}}\\s+O(s^3)\\ O(s^3)
		\end{bmatrix}\label{eqn:y_s}
		\end{align}
		Now we can calculate the distances involved in the statements. Firstly, $\|y_l-x\|=s+O(s^3)=d_M(x,y)+O(s^3)$ is a direct consequence of Equation \ref{eqn:y_l}.
		
		Accordingly to Equation \ref{eqn:y_s},  the only missing part to calculate $\|y_s-x\|$, is $\|y-c\|$. Recall that $c = x+r\sum_{i=1}^{D-d}\beta_i \boldsymbol{n_i}$, so
		\begin{align*}
		\|y-c\|&=\left[\sum_{i=1}^{D-d}\left\{\frac{\alpha_1}{2}s^2+O(s^3)-r\beta_i\right\}^2+s^2+O(s^4)\right]^{\frac{1}{2}}\\
		& = \left\{r^2+O(s^2)\right\}^{\frac{1}{2}}=r+O(s^2)
		\end{align*}
		As a result, 
		$$\|y_s-x\|=\frac{r}{\|y-c\|}\left\{s+O(s^3)\right\}=\frac{r}{r+O(s^2)}\left\{s+O(s^3)\right\}=s+O(s^3).$$
	\end{proof}

	\subsection{Global estimation error}
	
	In this section we prove the global error bound stated in Theorem \ref{thm:globalerror}.
	\begin{proof}[Proof of Theorem \ref{thm:globalerror}]
		Before proving the inequalities, we define the graph geodesic distance
		$$d_G(x,y)\coloneqq \min_P \sum_{i=0}^{p-1}d_M(x_i,x_{i+1}),$$
		where $P$ varies over all paths along $G$ with $x_0=x$ and $x_p=y$. Clearly we have $d_M(x,y)\leq d_G(x,y)$ for any $x,y\in M$ and any graph $G$. The idea is to show $d_SG(x,y)\approx d_G(x,y)\approx d_M(x,y)$. 
		
		First we prove the first inequality: $(1-\lambda_1)d_M(x,y)\leq d_{SG}(x,y)$. Assume $P=\{x_i\}_{i=1}^p$ minimizes $d_{SG}$. Then we have
		\begin{align*}
		d_SG(x,y)&=\sum_{i=0}^{p-1} d_{S}(x_i,x_{i+1})\geq \sum_{i=0}^{p-1} (1-C\epsilon^2_{\max})d_M(x_i,x_{i+1})\\
		&\geq (1-C\epsilon^2_{\max})\sum_{i=0}^{p-1}d_M(x_i,x_{i+1})\geq (1-C\epsilon^2_{\max}) d_G(x,y)\geq (1-\lambda_1) d_M(x,y),
		\end{align*}
		where $\lambda_1=C\epsilon^2_{\max}$. To prove the other inequality, assume $P=\{x_i\}_{i=0}^{p-1}$ minimizes $d_G$, then
		\begin{align*}
		d_{SG}(x,y)&\leq \sum_{i=0}^{p-1} d_S(x_i,x_{i=1})\leq \sum_{i=0}^{p-1} \left(1+C\epsilon^2_{\max}\right)d_M(x_i,x_{i+1})\\
		&= \left(1+C\epsilon^2_{\max}\right) \sum_{i=0}^{p-1}d_M(x_i,x_{i+1})= \left(1+C\epsilon^2_{\max}\right) d_G(x,y).
		\end{align*}
		By theorem 2 in \cite{geodist2000}, $d_G(x,y)\leq (1+\frac{4\delta}{\epsilon_{\min}})d_M(x,y)$. Combining the above two equalities we conclude that  $d_{SG}(x,y)\leq (1+\lambda_2)d_M(x,y)$, where $\lambda_2= \frac{4\delta}{\epsilon_{\min}}+C\epsilon_{\max}^2+\frac{4C\delta\epsilon^2_{\max}}{\epsilon_{\min}}$.
	\end{proof}
	
\end{document}